\documentclass[letterpaper, 10 pt, conference]{ieeeconf}  % Comment this line out if you need a4paper

\IEEEoverridecommandlockouts
\usepackage{graphics} % for pdf, bitmapped graphics files
\usepackage{epsfig} % for postscript graphics files
\usepackage{times} % assumes new font selection scheme installed

\usepackage[utf8]{inputenc} % allow utf-8 input
\usepackage[T1]{fontenc}    % use 8-bit T1 fonts
\usepackage{hyperref}       % hyperlinks
\usepackage{url}            % simple URL typesetting
\usepackage{booktabs}       % professional-quality tables
\usepackage{amsfonts}       % blackboard math symbols
\usepackage{nicefrac}       % compact symbols for 1/2, etc.
\usepackage{lipsum}
\usepackage[english]{babel}

\usepackage{amsmath,amssymb}
\usepackage{mathtools}
\usepackage{algorithm, algorithmic}
\usepackage{pifont}
\usepackage{cases}
\usepackage{stackengine}    % circled symbols

\usepackage[table,dvipsnames]{xcolor}

\newtheorem{theorem}{Theorem}[section]

\newtheorem{lemma}[theorem]{Lemma}
\newtheorem{assumption}[theorem]{Assumption}

\newcommand{\tild}{\widetilde}
\newcommand{\eps}{\varepsilon}

\newcommand{\ol}{\overline}
\newcommand{\onevec}{\mathbf{1}_n}

\newcommand{\cset}{\mathcal{C}}

\newcommand{\numberthis}{\addtocounter{equation}{1}\tag{\theequation}}

\DeclareMathOperator*{\argmin}{arg\,min}

\newcommand{\R}{\mathbb{R}}
\newcommand{\Z}{\mathbb{Z}}

\newcommand{\E}{\mathbb{E}}

\newcommand{\mU}{{\bf u}}
\newcommand{\mV}{{\bf v}}
\newcommand{\mW}{{\bf W}}
\newcommand{\mX}{{\bf x}}
\newcommand{\mY}{{\bf y}}

\newcommand{\aX}{\overline{\bf x}}

\newcommand{\aU}{\overline{\bf u}}
\newcommand{\aV}{\overline{\bf v}}

\newcommand{\cC}{{\mathcal{C}}}
\newcommand{\cD}{{\mathcal{D}}}

\newcommand{\bbf}{{\bf f}}

\newcommand{\bx}{{\bf x}}
\newcommand{\by}{{\bf y}}

\newcommand{\Lav}{L_{g}}
\newcommand{\Lmax}{L_{l}}
\newcommand{\Lworst}{L_{\xi}}
\newcommand{\muav}{\mu_{g}}
\newcommand{\mumin}{\mu_{l}}
\newcommand{\sigmaav}{\sigma_{g}}

\newcommand{\ds}{\displaystyle}
\newcommand{\norm}[1]{\left\| #1 \right\|}

\newcommand{\angles}[1]{\left\langle #1 \right\rangle}
\newcommand{\cbraces}[1]{\left( #1 \right)}
\newcommand{\sbraces}[1]{\left[ #1 \right]}
\newcommand{\braces}[1]{\left\{ #1 \right\}}

\newcommand{\pd}[1]{\textcolor{black}{#1}}
\newcommand{\rev}[1]{\textcolor{black}{#1}}

\title{\LARGE \bf
	An Accelerated Method For Decentralized Distributed Stochastic Optimization Over Time-Varying Graphs
}

\author{Alexander Rogozin, Mikhail Bochko, Pavel Dvurechensky, Alexander Gasnikov, Vladislav Lukoshkin% <-this % stops a space
	\thanks{
		A.R. and M.B. are with the Moscow Institute of Physics and Technology and HSE University, 
		Russian Federation 
		{\tt\small ([aleksandr.rogozin,bochko.mg]@phystech.edu)}. 
		P.D. is with the Weierstrass Institute for Applied Analysis and Stochastics, 
		Germany,
		Institute for Information Transmission Problems RAS, 
		HSE University, Russian Federation
		{(\tt\small pavel.dvurechensky@wias- berlin.de)}.
		A.G. is with  Moscow Institute of Physics and Technology, 
		Institute for Information Transmission Problems RAS, 
		HSE University, Russian Federation {(\tt\small gasnikov@yandex.ru)}.
		V.L. is with the Skolkovo Institute of Science and Technology,
		Russia {(\tt\small lukoshkin@phystech.edu)}. 
		The research was supported by the Ministry of Science and Higher Education of the Russian Federation (Goszadaniye) No. 075-00337-20-03, project No. 0714-2020-0005.
	}%
}

\begin{document}
	
	\maketitle
	\thispagestyle{empty}
	\pagestyle{empty}

	%%%%%%%%%%%%%%%%%%%%%%%%%%%%%%%%%%%%%%%%%%%%%%%%%%%%%%%%%%%%%%%%%%%%%%%%%%%%%%%%
	\begin{abstract}
		We consider a distributed stochastic optimization problem that is solved by a decentralized network of agents with only local communication between neighboring agents.
		The goal of the whole system is to minimize a global objective function given as a sum of local objectives held by each agent. Each local objective is defined as an expectation of a convex smooth random function and the agent is allowed to sample stochastic gradients for this function. For this setting we propose \pd{the first} accelerated (in the sense of Nesterov's acceleration) method that simultaneously attains optimal up to a logarithmic factor communication and oracle complexity bounds for smooth strongly convex distributed stochastic optimization. We also consider the case when the communication graph is allowed to vary with time and obtain complexity bounds for our algorithm, which are the first upper complexity bounds for this setting in the literature.
	\end{abstract}

	\section{Introduction}
	Distributed algorithms have already about half a century history \cite{bor82,tsi84,deg74} with many applications including robotics, resource allocation, power system control, control of drone or satellite networks, distributed statistical inference and optimal transport, multiagent reinforcement learning \cite{xia06,rab04,ram2009distributed,kra13,ned17e,nedic2017fast,uribe2018distributed,kroshnin2019complexity,ivanova2020composite}.
	Recently, development of such algorithms has become one of the main topics in optimization and machine learning motivated by large-scale learning problems with privacy constraints and other challenges such as data being produced or stored distributedly \cite{bot10,boy11,aba16,ned16w,ned15}. An important part of this research studies decentralized distributed optimization algorithms over arbitrary networks. In this setting a network of computing agents, e.g. sensors or computers, is represented by a connected graph in which two agents can communicate with each other if there is an edge between them. This imposes communication constraints and the goal of the whole system~\cite{ned09,ram10,daneshmand2021newton} is to cooperatively minimize a global objective using only local communications between agents, each of which has access only to a local piece of the global objective. Due to random nature of the optimized process or randomness and noise in the used data, a particular important setting is distributed stochastic optimization. Moreover, the topology of the network can vary in time, which may prevent fast convergence of an algorithm.  
	%\todo{Michael: insert }
	
	More precisely, we consider the following optimization problem
	\begin{align}\label{eq:initial_problem}
	\min_{x\in\R^d} \left[f(x) := \frac{1}{n}\sum_{i=1}^n f_i(x)\right],   \quad f_i(x):=\E_{\xi_i\sim\cD_i} \bbf_i(x, \xi_i),
	\end{align}
	where $\xi_i$'s are random variables with probability distributions $\cD_i$. For each $i=1,...,n$ we make the following assumptions: $f_i(x)$ is a convex function and that almost sure w.r.t. distribution $\cD_i$, the function $\bbf_i(x, \xi_i)$ has gradient $\nabla \bbf_i(x, \xi_i)$, which is $L_i(\rev{\xi_i})$-Lipschitz continuous with respect to the Euclidean norm. 
	Further, for each $i=1,...,n$, we assume that we know a constant $L_i\geqslant 0$ such that $\sqrt{\E_{\rev{\xi_i}} L_i(\rev{\xi_i})^2 } \leq L_i < +\infty$. Under these assumptions, $\E_{\rev{\xi_i}}\nabla \bbf_i(x, \xi_i) = \nabla f_i(x)$ and $f$  is $L_i$-smooth, i.e. has $L_i$-Lipschitz continuous gradient with respect to the Euclidean norm. Also, we assume that, for all $x$, and $i$
	\begin{equation}
	\label{stoch_assumption_on_variance}
	\E_{\rev{\xi_i}}[\norm{\nabla \bbf_i(x, \xi_i) - \nabla f_i(x)}^2] \leqslant \sigma_i^2,
	\end{equation}
	where $\norm{\cdot}$ is the Euclidean norm. Finally, we assume that each $f_i$ is $\mu_i$-strongly convex \rev{($\mu_i > 0$)}. Important characteristics of the objective in \eqref{eq:initial_problem} are local strong convexity parameter $\ds \mumin = \min_i\mu_i$ and local smoothness constant $\Lmax = \max_i L_i$, which define local condition number $\kappa_l=\Lmax/\mumin$, as well as their global counterparts $\muav = \frac{1}{n}\sum_{i=1}^n\mu_i,~ \Lav = \frac{1}{n}\sum_{i=1}^n L_i$, $\kappa_g=\Lav /\muav$.
	The global condition number may be significantly better than local (see e.g. \cite{scaman2017optimal} for details) and it is desired to develop algorithms with complexity depending on the global condition number. Moreover, we introduce a worst-case smoothness constant over stochastic realizations $\ds \Lworst = \max_i \max_{\xi} L_i(\xi)$ and a maximum gradient norm at optimum $\ds M_\xi = \max_i \max_{\xi} \norm{\nabla f_i(x^*, \xi)}$ and assume that these constants are well-defined \rev{(finite)}. Similarly to global smoothness and strong convexity constants, we introduce $\sigmaav^2 = \frac{1}{n}\sum_{i=1}^n \sigma_i^2$.
	%Our analysis shows that performance of Algorithm \ref{alg:decentralized_agd} depends on global constants.

	To introduce the distributed optimization setup, we assume that \rev{communication constraints in the computational network are represented by an undirected communication graph which may vary with time. Namely, the network is modeled with a sequence of graphs $\braces{\mathcal{G}^k = (V, E^k)}_{k=0}^\infty$. We note that the set of vertices remains the same, while set of edges is allowed to change with time.} Each agent in the network corresponds to a graph vertex and communication \rev{at time slot $k$} is possible only between nodes which are connected by an edge \rev{at this time slot}. Further, each agent $i$ has access only to iid samples from $\cD_i$ and corresponding stochastic gradients $\nabla \bbf_i(x, \xi_i)$. The goal of the whole network is to solve the minimization problem \eqref{eq:initial_problem} by using only communication between neighboring nodes. The performance of decentralized optimization algorithms depend\rev{s} on the characteristic number $\chi$ of the network that quantifies its connectivity and how fast the information is spread over the network. The precise definition will be given later.

	\subsection{Related work}
	In centralized setting optimal methods exist \cite{ghadimi2013optimal} for the considered setting of smooth strongly convex stochastic optimization, as well as many algorithms for other settings \cite{dvurechensky2018parallel,Zhang2018,dvurechensky2021hyperfast,agafonov2021accelerated}.
	%For the considered setting of smooth strongly convex stochastic optimization problems optimal methods exist \pd{\cite{ghadimi2013optimal}} in centralized setting, which is equivalent to standard optimization. 
	Decentralized distributed optimization introduces several challenges, one of them being that one has to care about two complexities: number of oracle calls which are made by each agent and the number of communication steps, which are sufficient to reach a given accuracy $\eps$. In the simple case of all constants $\mu_i,L_i,\sigma_i$ being independent on $i$, the oracle complexity lower bound \cite{scaman2017optimal,dvinskikh2019decentralized}
	%\todo{this bound do not account for different constants for different $f_i$'s} 
	$\Omega\left(\max\left\{\frac{\sigma^2}{n\mu\eps},\sqrt{\frac{L}{\mu}}\ln\frac{1}{\eps}\right\}\right)$ is a clear counterpart of the \rev{centralized} lower bound \cite{Nemirovskii1983}. The lower bound on communications number $\Omega\left(\sqrt{\frac{L}{\mu}\chi}\ln\frac{1}{\eps} \right)$ corresponds to decentralized deterministic optimization and, compared to standard non-distributed accelerated methods \cite{nesterov1983method,nesterov2020primal-dual,guminov2019accelerated}, has an additional network-dependent factor $\sqrt{\chi}$. Existing distributed algorithms  \cite{fallah2019robust,dvinskikh2019primal,dvinskikh2019decentralized,olshevsky2019asymptotic,olshevsky2019non} achieve either the lower oracle complexity bound or the lower communication complexity bound, but not both simultaneously. In this paper we propose an algorithm which closes this gap and achieves both bounds simultaneously.  
	%\todo{is this true?}
	
	Deterministic decentralized \rev{optimization} is quite well understood with many centralized algorithms having their decentralized counterparts. For example, there are decentralized subgradient method \cite{Nedic2009}, gradient methods \cite{shi2015extra,rogozin2019projected} and many variants of accelerated gradient methods \cite{Qu2017,ye2020multi,li2018sharp,Jakovetic,dvu2018,dvinskikh2019decentralized,li2020revisiting}, which achieve both communication and \rev{oracle} complexity lower bounds \cite{hendrikx2020optimal,li2020optimal,scaman2017optimal,tang2019practicality}. 
	The negative side of the majority of the accelerated distributed methods is that their complexity depends on the local condition number $\kappa_l$, which may be larger than the global condition number $\kappa_g$, which corresponds to the centralized optimization. A number of methods \cite{scaman2017optimal,Wu2017,Zhang2017,dvu2018,uribe2020dual} require an assumption that the Fenchel conjugate for each $f_i(x)$ is available, which may be restrictive in practice.
	In contrast, our complexity bounds depend on the global condition number and we use the primal approach without additional assumptions on $f_i$'s.
	
	Another important part of our paper is decentralized distributed optimization on time-varying networks. This area is not understood as well as the simpler setting of fixed networks. The first method with provable geometric convergence was proposed in \cite{Nedic2017achieving}. Such primal algorithms as Push-Pull Gradient Method \cite{Pu2018} and DIGing \cite{Nedic2017achieving} are robust to network changes and have theoretical guarantees of convergence over time-varying graphs. Recently, a dual method for time-varying architectures was introduced in \cite{Maros2018}. All these methods do not allow to achieve the above lower bounds. \pd{Moreover, we are not aware of any accelerated algorithms for stochastic optimization on time-varying networks.}

	\subsection{Our contributions}
	\pd{
		The main contribution of this paper is twofold. Firstly, when the communication network is fixed, we propose the first optimal up to logarithmic factors accelerated decentralized distributed algorithm for stochastic convex optimization. This means that our algorithm has
		oracle per node complexity $\widetilde{O}\left(\max\left\{\frac{\sigma_g^2}{n\mu_g\eps},\sqrt{\frac{L_g}{\mu_g}}\ln\frac{1}{\eps}\right\}\right)$ and communication complexity $\widetilde{O}\left(\sqrt{\frac{L_g}{\mu_g}\chi}\ln\frac{1}{\eps} \right)$.
		Importantly, our communication bound depends on global constants $L_g,\mu_g$ whereas existing algorithms, even for deterministic setting, provide bounds which depend on local constants $L_l$, $\mu_l$ that can be much worse than $L_g$, $\mu_g$.
	}
	
	\pd{
		Secondly, we propose the first accelerated distributed stochastic optimization algorithm over time-varying graphs. This algorithm has the same oracle per node complexity as the above algorithm and the communication complexity $\widetilde{O}\left(\frac{\tau}{\lambda}\sqrt{\frac{L_g}{\mu_g}}\ln\frac{1}{\eps} \right)$, where $\tau$ and $\lambda$ characterize the dynamics of the communication graph (see the precise definition in Assumption \ref{assum:mixing_matrix}).
	}

	\section{Preliminaries}
	
	\subsection{Problem reformulation}
	
	In order to solve problem \eqref{eq:initial_problem} in a decentralized manner, we assign a local copy of $x$ to each node in the network, which leads to a linearly constrained problem
	\begin{align}\label{eq:problem_inflated}
	\min_{\bx\in\R^{n\times d}}~ F(\bx) = \sum_{i=1}^m f_i(x_i)~ \text{s.t.}~ x_1 = \ldots = x_n,
	\end{align}
	where $\bx = (x_1\ldots x_n)^\top\in\R^{n\times d}$. We denote the feasible set $\cC = \braces{x_1 = \ldots = x_n}$ for brevity. Strong convexity and smoothness parameters of $F$ are related to that of functions $f_i$. Namely, $F$ is $\Lmax$-smooth and $\mumin$-strongly convex on $\R^{n\times d}$ and $\Lav$-smooth and $\muav$-strongly convex on the set $\cC$.

	\subsection{Consensus procedure}
	
	In this subsection we discuss, how the agents can interact by exchanging information. Importantly, the communication  graph $\mathcal{G}$ can change with time. Thus, we consider a sequence of undirected communication graphs $\braces{\mathcal{G}^k = (V, E^k)}_{k=0}^\infty$ and a sequence of corresponding mixing matrices $\braces{\mW^k}_{k=0}^\infty$ associated with it. We impose the following 
	\begin{assumption}\label{assum:mixing_matrix}
		Mixing matrix sequence $\braces{\mW^k}_{k=0}^\infty$ satisfies the following properties.
		\begin{itemize}
			\item (Decentralized property) $(i, j)\notin E_k \;\Rightarrow \;[\mW^k]_{ij} = 0$.
			%  		\item (Decentralized property) If $(i, j)\notin E_k$, then $[\mW^k]_{ij} = 0$.
			\item (Double stochasticity) $\mW^k \onevec = \onevec,~ \onevec^\top\mW^k = \onevec^\top$.
			\item (Contraction property) There exist $\tau\in\Z_{++}$ and $\lambda\in(0, 1)$ such that for every $k\ge \tau - 1$ it holds
			\begin{align*}
			\norm{\mW_{\tau}^k \mX - \aX} \le (1 - \lambda)\norm{\mX - \aX}, 
			\end{align*}
			where $\mW_\tau^k = \mW^k \ldots \mW^{k-\tau+1}$.
		\end{itemize}
	\end{assumption}
	
	The contraction property in Assumption \ref{assum:mixing_matrix} was initially proposed in \cite{koloskova2020unified} in a stochastic form. This property generalizes several assumptions in the literature.
	\begin{itemize}
		\item Time-static connected graphs. In this scenario we have $\mW^k = \mW$. Therefore, $\lambda = 1 - \lambda_2(\mW)$, where $\lambda_2(\mW)$ denotes the second largest eigenvalue of $\mW$.
		\item Sequence of connected graphs: every $\mathcal{G}_k$ is connected. In this case $\lambda = 1 - \underset{k\ge 0}{\sup}~\lambda_2(\mW^k)$.
		\item $\tau$-connected graph sequence: for every $k\ge 0$ graph $\mathcal{G}^k_\tau = (V, E^k\cup E^{k+1}\cup\ldots\cup E^{k+\tau-1})$ is connected \cite{Nedic2017achieving}. For $\tau$-connected graph sequences it holds $1 - \lambda = \underset{k\ge 0}{\sup}~\lambda_{\max}(\mW_\tau^k - \frac{1}{n}\onevec\onevec^\top)$.
	\end{itemize}

	During every (synchronized) communication round, the agents pull information from their neighbors and update their local vectors according to the rule
	\begin{align*}
	x_i^{k+1} = [\mW^k]_{ii} \rev{x_i^k} + \sum_{(i, j)\in E^k} [\mW^k]_{ij} x_j^k,
	\end{align*}
	which writes as $\mX^{k+1} = \mW^k \mX^k$ in matrix form. The contraction property in Assumption \ref{assum:mixing_matrix} requires a specific choice of weights in $\mW^k$. Choosing Metropolis weighs is sufficient to ensure the contraction property for $\tau$-connected graph sequences (see \cite{Nedic2017achieving} for details):
	\begin{align*}
	[\mW^k]_{ij} = 
	\begin{cases}
	1 / (1 + \max\{d^k_i, d^k_j\}) &\text{if }(i, j)\in E^k, \\
	0 &\text{if } (i, j)\notin E^k, \\
	1 - \ds\sum_{(i, m)\in E^k} [\mW^k]_{im} &\text{if } i = j, 
	\end{cases}
	\end{align*}
	where $d^k_i$ denotes the degree of node $i$ in graph $\mathcal{G}^k$.
	
	\pd{When the communication graph $\mathcal{G}$ does not change with time, it is possible to apply accelerated consensus procedures by leveraging Chebyshev acceleration~\cite{auzinger2011iterative,scaman2018optimal}: given the reference matrix $\mW$ as above, set $W_T=P_T(\mW)$ and $P_T(1)=1$ (the latter is to ensure the double stochasticity of $W_T$), with $T$ being the number of consensus steps and $P_T$ being the Chebyshev polynomial of degree $T$. In this case one can guarantee that
		\begin{align*}
		\norm{W_T \mX - \aX} \le (1 - \sqrt{1-\rho})^T\norm{\mX - \aX}, 
		\end{align*}
		where $\rho :=  \lambda_2(\mW) < 1$. In this case we define $\chi=\frac{1}{1-\rho}$.
	}

	\section{Algorithm and main result}
	In this section we describe the proposed algorithm and give its convergence theorem. Our algorithms is an accelerated mini-batch stochastic gradient method equipped with a consensus procedure. Let $\braces{\xi_i^\ell}_{\ell=1}^r$ be independent random variables with distribution $\cD_i$. For function $\bbf_i$ we define its batched gradient of size $r$ as
	\begin{align*}
	\nabla^r \bbf\cbraces{x, \braces{\xi_i^\ell}_{\ell=1}^r} = \frac{1}{r} \sum_{\ell=1}^r \nabla \bbf_i(x, \xi_\ell^r).
	\end{align*}
	Batched gradient for $F(\bx)$ is defined analogously. Let $\braces{\xi^\ell}_{\ell=1}^r$ be independent, where $\xi^\ell = (\xi_1^\ell \ldots \xi_n^\ell)^\top$ is a random vector consisting of random variables at all nodes. Then we define $\nabla^r F(\bx, \braces{\xi}_{\ell=1}^r)$ as a matrix of $\R^{n\times d}$, the $i$-th row of which stores $\nabla^r \bbf_i\cbraces{x_i, \braces{\xi_i^\ell}_{\ell=1}^r}$.
	% \begin{align*}
	%     \nabla^r F(\bx, \braces{\xi}_{\ell=1}^r) = \sbraces{\nabla^r \bbf_1\cbraces{x_1, \braces{\xi_1^\ell}_{\ell=1}^r} \ldots \nabla^r \bbf_n\cbraces{x_n, \braces{\xi_n^\ell}_{\ell=1}^r}}^\top.
	% \end{align*}
	For brevity we use notation $\nabla^r \bbf_i(x_i),~ \nabla^r F(\bx)$ for batched gradients of $\bbf_i$ and $F$, respectively.
	
	To describe the algorithm we introduce sequences of extrapolation coefficients $\alpha^k, A^k$ similar to that of \cite{stonyakin2020inexact}, which are defined as follows.
	\begin{subequations}
		\begin{align}\label{eq:Ak_def}
		\alpha^0 &= A^0 = 0 \\
		(A^k + \alpha^{k+1})(1 + A^k\muav/2) &= 2\Lav(\alpha^{k+1})^2, \\
		A^{k+1} &= A^k + \alpha^{k+1}.
		\end{align}    
	\end{subequations}

	% \begin{algorithm}[H]
	%     \caption{Decentralized Stochastic AGD}
	%     \label{alg:decentralized_agd}
	%     \begin{algorithmic}[1]
	%         \REQUIRE{All nodes hold equal values $u_i^0 = x_i^0 = x^0$}
	%         \FOR{$k = 0, 1, 2,\ldots$}
	%             \STATE{Node $i$ computes $\ds y_i^{k+1} = \frac{\alpha^{k+1} u_i^k + A^k x_i^k}{A^{k+1}}$}
	%         	\STATE{\label{alg_step:agd_step} Node $i$ computes $\ds v_i^{k+1} = \frac{\mu y_i^{k+1} + (1 + A^k\mu) u_i^k}{1 + A^k\mu + \mu} - \frac{\alpha^{k+1}\nabla^r f_i(y_i^{k+1})}{1 + A^k\mu + \mu}$ \todo{define $\nabla^r$, also stochastic should be here.}}
	%         	\STATE{\label{alg_step:consensus_update} Nodes run consensus procedure on $v_i^{k+1}$ for $T^k$ iterations and obtain $u_i^{k+1}$}
	%         	\STATE{Node $i$ computes $\ds x_i^{k+1} = \frac{\alpha^{k+1} u_i^{k+1} + A^k x_i^k}{A^{k+1}}$}
	%         \ENDFOR
	%     \end{algorithmic}
	% \end{algorithm}

	\begin{algorithm}[H]
		\caption{Decentralized Stochastic AGD}
		\label{alg:decentralized_agd}
		\begin{algorithmic}[1]
			\REQUIRE{Initial guess $\mX^0\in \cset$, constants $\Lav, \muav > 0$, \\$\mU^0 = \mX^0$}
			\FOR{$k = 0, 1, 2,\ldots$}
			\STATE{\hspace{-0.25cm}$\mY^{k+1} = \frac{\alpha^{k+1} \mU^k + A^k \mX^k}{A^{k+1}}$}
			\STATE{\vspace{0.05cm}\hspace{-0.25cm}\label{alg_step:agd_step}$\mV^{k+1} = \frac{(\alpha^{k+1}\muav/2)\mY^{k+1} + (1 + A^k\muav/2)\mU^k}{1 + A^{k + 1}\muav/2} - \frac{\alpha^{k+1}\nabla^r  F(\mY^{k+1})}{1 + A^{k + 1}\muav/2}$}
			\STATE{\label{alg_step:consensus_update}
				$
				\hspace{-0.25cm}
				\mU^{k+1} = \text{Consensus}(\mV^{k+1}, T^k)
				$
			}
			\STATE{\hspace{-0.25cm}$\mX^{k+1} = \frac{\alpha^{k+1} \mU^{k+1} + A^k \mX^k}{A^{k+1}}$}
			\ENDFOR
		\end{algorithmic}
	\end{algorithm}
	\vspace{-0.5cm}
	\begin{algorithm}[H]
		\caption{Consensus}
		\label{alg:consensus}
		\begin{algorithmic}
			\REQUIRE{Initial $\mX^0\in\cset$, number of iterations $T$.}
			\FOR{$t = 1, \ldots, T$}
			\STATE{$\mX^{t+1} = \mW^t \mX^t$}
			\ENDFOR
		\end{algorithmic}
	\end{algorithm}

	In the next theorem, we provide oracle and communication complexities of Algorithm \ref{alg:decentralized_agd}, i.e. we estimate the number of stochastic oracle calls by each node and the number of communication rounds to solve problem \eqref{eq:initial_problem} with accuracy $\eps$.
	\begin{theorem}[Main result]\label{th:total_iterations_strongly_convex}
		\pd{Let $\eps > 0$ be the desired accuracy. Set}
		\begin{align*}
		T_k = T = \frac{\tau}{2\lambda}\log\frac{D}{\delta'},~ \delta' = \frac{n\eps}{32} \frac{\muav^{3/2}}{\Lav^{1/2} \Lmax^2},~
		r = \frac{2\sigma_{\pd{g}}^2}{\eps\sqrt{\Lav\muav}},
		\end{align*}
		where
		\begin{align}\label{eq:def_sqrt_D}
		&\sqrt D = \cbraces{\frac{2\Lmax}{\sqrt{\Lav\muav}} + 1}\sqrt{\delta'} + \frac{2nM_\xi}{\sqrt{\Lav\muav}} \\
		&+\frac{2\Lmax}{\muav} \sqrt{n} \cbraces{\norm{\ol u^0 - x^*}^2 + \frac{2}{\sqrt{\Lav\muav}}\cbraces{\frac{\sigmaav^2}{4n\Lav r^2} + \delta}}^{1/2} \notag.
		\end{align}
		\pd{Then, to yield $\mX^N$ such that
			\begin{align*}
			&\E f(\ol x^N) - f(x^*)\le \eps,~ \E\norm{\mX^N - \aX^N}^2\le \delta'=O(\eps),
			\end{align*}
			Algorithm \ref{alg:decentralized_agd} requires \pd{no more than}
			\begin{align}\label{eq:agd_computational_complexity}
			N_{\pd{orcl}} = N\cdot r = \dfrac{6\sigmaav^2}{n\muav\eps} \log\cbraces{\frac{4\Lav\norm{\ol u^0 - x^*}^2}{\eps}}
			\end{align}
			stochastic oracle calls at each node and \pd{no more than}
			\begin{align}\label{eq:agd_communication_complexity}
			N_{\text{comm}} 
			&= 3\sqrt{\frac{\Lav}{\muav}} \kappa \cdot \log\cbraces{\frac{4\Lav\norm{\ol u^0 - x^*}^2}{\eps}} \log\frac{D}{\delta'},
			\end{align}
			\rev{communication rounds}, where $\kappa= \frac{\tau}{2\lambda} $ under Assumption \ref{assum:mixing_matrix} and  $\kappa=\sqrt{\chi}$ 
			% 	\pd{no more than}
			% 	\begin{align}\label{eq:agd_communication_complexity}
			%     	N_{\text{comm}} 
			%         &= \sqrt{\frac{\Lav}{\muav}\chi} \cdot \log\cbraces{\frac{4\Lav\norm{\ol u^0 - x^*}^2}{\eps}} \log\frac{D}{\delta'}
			% 	\end{align}
			% 	communication steps 
			when the communication graph is static.}
	\end{theorem}
	We provide the proof of Theorem \ref{th:total_iterations_strongly_convex} in Section \ref{sec:analysis}.
	
	\pd{The number of stochastic oracle calls at each node in \eqref{eq:agd_computational_complexity} coincides with the lower bound for centralized optimization up to a constant factor. When the graph is time-varying, the number of communication steps includes an additional factor $\tau/\lambda$, which characterizes graph connectivity. If the communication graph is fixed, \rev{in addition to the lower oracle complexity bound, our algorithm also achieves lower communication bound up to a polylogarithmic factor.}}

	\section{Analysis of the algorithm}\label{sec:analysis}
	\pd{Analysis of our algorithm consists of three main parts. Firstly, if an approximate consensus is imposed on local variables at each node, this ensures a stochastic inexact oracle for the global objective $f$. Secondly, we analyze an accelerated stochastic gradient method with stochastic inexact oracle. Thirdly, we analyze, how the consensus procedure allows to obtain an approximate consensus. Finally, we combine the building blocks together and prove the main result.}
	\subsection{Stochastic inexact oracle via inexact consensus}
	In this subsection we show that if a point $\mX\in\R^{n\times d}$ is close to the set $\cset$, i.e. it approximately satisfies consensus constraints, then, the mini-batched and averaged among nodes stochastic gradient provides a stochastic inexact oracle developed in \pd{\cite{devolder2013first,dvurechensky2016stochastic,gasnikov2016stochasticInter}}.

	Consider $\ol x, \ol y\in\R^{d}$ and define $\ol\bx = \onevec\ol x^\top = (\ol x\ldots \ol x)^\top,~ \ol\by = \onevec\ol y^\top = (\ol y\ldots \ol y)^\top\in \R^{n\times d}$. Let $\mX\in\R^{n\times d}$ be such that $\Pi_\cset (\mX) = \ol\mX$ and $\norm{\ol\mX - \mX}^2\le \delta'$.
	
	\begin{lemma}\label{lem:inexact_oracle}
		Define
		\begin{align*}
		\delta &= \frac{1}{2n}\cbraces{\frac{\Lmax^2}{\Lav} + \frac{2\Lmax^2}{\muav} + \Lmax - \mumin} \delta', \numberthis\label{eq:delta_inexact_oracle} \\
		f_{\delta, L, \mu}(\ol x, \mX) &= \frac{1}{n} \sbraces{F(\mX) + \angles{{\nabla} F(\mX), \ol\mX - \mX}} \\
		&\quad + \cbraces{\frac{\mumin}{2n} - \frac{2\Lmax^2}{2n\muav}}\norm{\ol\mX - \mX}^2, \\
		g_{\delta, L, \mu}(\ol x, \mX) &= \frac{1}{n}\sum_{i=1}^n \nabla f_i(x_i)\\
		\tild{g}_{\delta, L, \mu}(\ol x, \mX) &= \frac{1}{n}\sum_{i=1}^n  \frac{1}{r}\sum_{j=1}^r \nabla\bbf_i(x_i,\xi_i^j).
		\end{align*}
		Firstly, for any $\ol y\in\R^d$ it holds
		%$(f_{\delta,L,\mu}(\ol x, \mX), g_{\delta,L,\mu}(\ol x, \mX))$ satisfies
		\begin{align*}%\label{eq:inexact_oracle_f}
		\hspace{-0.25cm}\frac{\muav}{4}\norm{\ol y - \ol x}^2 &\le f(\ol y) - f_{\delta, L, \mu}(\ol x, \mX) - \angles{ g_{\delta, L, \mu}(\ol x, \mX), \ol y - \ol x} \\
		\hspace{-0.25cm}\Lav\norm{\ol y - \ol x}^2 + \delta &\ge f(\ol y) - f_{\delta, L, \mu}(\ol x, \mX) - \angles{ g_{\delta, L, \mu}(\ol x, \mX), \ol y - \ol x}.
		\end{align*}
		Secondly, $\tild{g}_{\delta, L, \mu}(\ol x, \mX)$ satisfies
		\begin{subequations}\label{eq:tilde_g_properties}
			\begin{align}
			\E \tild{g}_{\delta, L, \mu}(x) &= g_{\delta, L, \mu}(x) \label{eq:tilde_g_properties_1} \\
			\E \|\tild{g}_{\delta, L, \mu}(\ol x, \mX) - g_{\delta, L, \mu}(\ol x, \mX)\|^2 &\leq \frac{\sum_{i=1}^n\sigma^2_i}{n^2r}\pd{=
				\frac{\sigma_g^2}{nr}}.
			\label{eq:tilde_g_properties_2}
			\end{align}
		\end{subequations}
	\end{lemma}
	
	\vspace{0.1cm}
	The first statement is proved in Lemma 2.1 of \cite{rogozin2020towards}; the proof of the second statement is provided in Appendix \ref{app:inexact_oracle}.

	\subsection{Similar Triangles Method with Stochastic Inexact Oracle}

	In this subsection we present a general algorithm for minimization problems with stochastic inexact oracle. This subsection is independent from the others and generalizes the algorithm and analysis from \cite{stonyakin2020inexact,stonyakin2019gradient} to the stochastic setting.
	Let $f(x)$ be a convex function defined on a convex set $Q\subseteq\R^m$. We assume that $f$ is equipped with stochastic inexact oracle having two components. The first component $(f_{\delta,L,\mu}(x), g_{\delta,L,\mu}(x))$ exists at any point $x\in Q$ and satisfies
	\begin{align}\label{eq:inexact_oracle_def_devolder}
	\frac{\mu}{2}\norm{y - x}^2 
	&\le f(y) - \cbraces{f_{\delta,L,\mu}(x) + \angles{g_{\delta,L,\mu}(x), y - x}} \notag\\
	&\le \frac{L}{2} \norm{y - x}^2 + \delta
	\end{align}
	for all $y\in Q$. To allow more flexibility, we assume that $\delta$ may change with the iterations of the algorithm.
	The second component $\tild{g}_{\delta,L,\mu}(x)$ is stochastic, is available at any point $x\in Q$, and  satisfies 
	\begin{align}
	\label{eq:inexact_oracle_moments}
	\E \tild{g}_{\delta, L, \mu}(x) = g_{\delta, L, \mu}(x), \quad \E \|\tild{g}_{\delta, L, \mu}(x) - g_{\delta, L, \mu}(x)\|^2 \leq \sigma^2. 
	\end{align}
	We also denote the \textit{batched} version of the stochastic component as
	\begin{align}
	\label{eq:inexact_oracle_batched_def}
	\tild{g}^{\ r}_{\delta, L, \mu}(x) = \frac{1}{r} \sum_{i=1}^r \tild{g}_{\delta, L, \mu}(x, \xi_i),
	\end{align}
	where $\xi_i$'s are iid realizations of the random variable $\xi$. It is straightforward that 
	\begin{subequations}\label{eq:inexact_oracle_batched_prop}
		\begin{align}
		&\E \tild{g}^{\ r}_{\delta, L, \mu}(x) = g_{\delta, L, \mu}(x), \\ &\E \|\tild{g}^{\ r}_{\delta, L, \mu}(x) - g_{\delta, L, \mu}(x)\|^2 \leq \frac{\sigma^2}{r}.
		\end{align}
		
	\end{subequations}
	
	Let us consider the following algorithm for minimizing $f$. Note that the error $\delta$ of the oracle and batch size $r$  may depend on the iteration counter $k$. Moreover, we let $\delta$ be stochastic.
	% \begin{algorithm}[H]
	\begin{algorithm}[ht]
		\caption{AGD with stochastic inexact oracle}
		\label{alg:FastAlg2_strong}
		\begin{algorithmic}[1]
			\REQUIRE{Initial guess $x^0$, constants $L, \mu \geq 0$, sequence of batch sizes $\{r_k\}_{k\geq 0}$.}
			\par Set $y^0:=x^0, u^0:=x^0, \alpha^0:=0, A^0:=0$
			\FOR{$k \geq 0$}
			\STATE{Find $\alpha^{k+1}$ as the greater root of: 
				$$(A^k + \alpha^{k+1})(1 + A^k \mu) = L \left( \alpha^{k+1} \right)^2$$ } 
			\STATE{\vspace{-0.5cm} Renew the following variables:
				\begin{gather*}
				A^{k+1}:=A^k + \alpha^{k+1} \\
				y^{k+1}:=\frac{\alpha^{k+1}u^k + A^k x^k}{A^{k+1}} 
				\end{gather*}
			}
			\STATE{\vspace{-0.5cm}
				Define the function:
				\begin{align*}
				&\phi^{k+1}(x) := \alpha^{k+1} \angles{\tild{g}^{\ r_{k+1}}_{\delta_{k+1}, L, \mu}(y^{k+1}),x - y^{k+1}} \\ &\quad + (1 + A^k \mu) \norm{x - u^k}^2 + \alpha^{k+1} \norm{x - y^{k+1}}^2
				\end{align*}
			}
			\STATE{\vspace{-0.5cm} Solve the optimization problem:
				$$u^{k+1} := \argmin_{x \in Q} \phi^{k+1}(x)$$}
			\STATE{\vspace{-0.5cm} Update $x$:
				$$x^{k+1}:=\frac{\alpha^{k+1} u^{k+1} + A^k x^k}{A^{k+1}}$$
			}
			\ENDFOR
		\end{algorithmic}
	\end{algorithm}
	
	We analyze convergence of Algorithm \ref{alg:FastAlg2_strong} by revisiting the proof of Theorem 3.4 in \cite{stonyakin2020inexact} and formulate the result in Theorem \ref{Th:fast_str_conv_adap} below. The complete proof is provided in Appendix \ref{app:fast_str_conv_adap}.
	\begin{theorem}
		\label{Th:fast_str_conv_adap}
		Let Algorithm \ref{alg:FastAlg2_strong} be applied to solve the problem $\min_{x\in Q} f(x)$. \pd{Let also $\norm{u^0 - x^*}\leq R$.} Then, after $N$ iterations we have
		%Then, after $N$ iterations of Algorithm \ref{alg:FastAlg2_strong} with batch size $r$  satisfying $\sigma^2/r=\eps L \alpha_k/A_k$, we have
		{\small
			\begin{align}
			%&\E f(x^N) - f(x^*) \leq \frac{\|u^0-x^*\|_2^2}{A_N} + \frac{2\sum_{k=0}^{N-1}A^{k+1}\delta}{A_N} + \frac{\eps}{2}, \\
			%+ \frac{\sum_{k=0}^{N-1}\widetilde{\delta}_k}{A_N}, \label{Th:fast_str_conv_adap:result_1} \\
			&\mathbb{E} f(x^N) - f(x^*) \leq \frac{1}{A^N} \left( \pd{R^2} +  \sum_{i=1}^N A^i \left( \frac{\sigma^2}{2Lr_{i}} + \E\delta_i \right) \right)\\
			%\E f(x^N) - f(x^*) \leq \frac{1}{A^N} &\left( \norm{u^0 - x^*}^2 + \left( \frac{\sigma^2}{2Lr} + \delta\right) \sum_{i=1}^N A^i \right) \\
			%&\E\|u^N-x^*\|_2^2\leq \frac{\|u^0-x^*\|_2^2}{1 + A^N\mu} + \frac{2\sum_{k=0}^{N-1}A^{k+1}\delta}{1 + A^N\mu } 
			%+ \frac{\sum_{k=0}^{N-1}\widetilde{\delta}_k}{(1 + A_N\mu + \att{A_Nm})}. \label{Th:fast_str_conv_adap:result_2}
			&\E \norm{u^N - x^*}^2 \leq \frac{1}{1 + A^N \mu} \left(\pd{R^2} + \sum_{i=1}^N A^i \left( \frac{\sigma^2}{2Lr_{i}} + \E\delta_i \right) \right)
			\end{align}
		}
	\end{theorem}
	
	In order to establish the rate, we recall the results of Lemma 5 in \cite{devolder2013first} and Lemma 3.7 in \cite{stonyakin2020inexact} and estimate the growth of coefficients $A^N$.
	\begin{lemma}
		\label{lem:Ak_properties}
		Coefficient $A^N$ can be lower-bounded as following: \rev{$A^N \geq 1/L\cdot\left( 1 + (1/2)\sqrt{{\mu}/L} \right)^{2(N-1)}$}.
		% \vspace{-0.3cm}
		% $$
		% A^N \geq \frac{1}{L} \left( 1 + \sqrt{\frac{\mu}{2L}} \right)^{2(N-1)}.
		% $$
		Moreover, we have \rev{$\sum_{i=1}^N A^i / A^N \le {1 + \sqrt{{L}/{\mu}}}$}.
		% \vspace{-0.2cm}
		% \begin{align*}
		%     \frac{\sum_{i=1}^N A^i}{A^N} \le {1 + \sqrt\frac{L}{\mu}}.
		% \end{align*}
	\end{lemma}

	\subsection{Proof of the main result}
	
	Thoughout this section, we denote $L = 2\Lav,~ \mu = \muav/2$ and $\sigma^2 = \sigmaav^2/(nr)$.

	\subsubsection{Outer loop}
	
	\begin{lemma}\label{lem:inexact_agd_convergence}
		Provided that consensus accuracy is $\delta'$, i.e. $\E\norm{\mU^{j} - \ol\mU^j}^2\le \delta' \text{ for } j = 1, \ldots, k$, we have
		{\small
			\begin{align}
			&\E f(\ol x^k) - f(x^*) \le \frac{1}{A^k}\cbraces{\norm{\ol u^0 - x^*}^2 + \cbraces{\frac{\sigma^2}{2Lr} + \delta} \sum_{i=1}^k A^i} \label{eq:inexact_agd_function_residual} \\
			&\E\norm{\ol u^k - x^*}^2 \le \frac{1}{1 + A^k\mu} \cbraces{\norm{\ol u^0 - x^*}^2 + \cbraces{\frac{\sigma^2}{2Lr} + \delta} \sum_{i=1}^k A^i} \nonumber
			\end{align}
		}
		where $\delta$ is given in \eqref{eq:delta_inexact_oracle}.
	\end{lemma}
	\begin{proof}
		First, assuming that $\E\norm{\mU^{j} - \ol\mU^j}^2\le \delta'$, we show that $\mY^j, \mU^j, \mX^j$ lie in $\sqrt{\delta'}$-neighborhood of $\cset$ by induction. At $j=0$, we have $\norm{\mX^0 - \ol\mX^0} = \norm{\mU^0 - \ol\mU^0} = 0$. Using $A^{j+1} = A^j + \alpha^j$, we get an induction pass $j\to j+1$.
		\begin{align*}
		&\E\norm{\mY^{j+1} - \ol\mY^{j+1}} \\
		&\qquad \le \frac{\alpha^{j+1}}{A^{j+1}}\E\norm{\mU^j - \ol\mU^j} + \frac{A^j}{A^{j+1}} \E\norm{\mX^j - \ol\mX^j}\le \sqrt{\delta'}, \\
		&\E\norm{\mX^{j+1} - \ol\mX^{j+1}} \\
		&\qquad \le \frac{\alpha^{j+1}}{A^{j+1}}\E\norm{\mU^{j+1} - \ol\mU^{j+1}} + \frac{A^j}{A^{j+1}} \E\norm{\mX^j - \ol\mX^j}\le \sqrt{\delta'}.
		\end{align*}
		Therefore, $g(\ol y) = \frac{1}{n}\sum_{i=1}^n \nabla f(y_i)$ represents the inexact gradient of $f$, and the desired result directly follows from Theorem \ref{Th:fast_str_conv_adap}.
	\end{proof}
	
	\subsubsection{Consensus subroutine iterations}
	In order to establish communication complexity of Algorithm \ref{alg:FastAlg2_strong}, we estimate the number of consensus iterations in the following Lemma.
	\begin{lemma}\label{lem:consensus_iters_strongly_convex}
		Let consensus accuracy be maintained at level $\delta'$, i.e. $\E\norm{\mU^j - \ol\mU^j}^2\le \delta' \text{ for } j = 1, \ldots, k$ and let Assumption \ref{assum:mixing_matrix} hold. Then it is sufficient to make $T_k = T = \frac{\tau}{2\lambda}\log\frac{D}{\delta'}$ consensus iterations, where $D$ is defined in \eqref{eq:def_sqrt_D}, in order to ensure $\delta'$-accuracy on step $k+1$, i.e. $\E\norm{\mU^{k+1} - \ol\mU^{k+1}}^2\le \delta'$.
	\end{lemma}
	Lemma \ref{lem:consensus_iters_strongly_convex} is analogous to Lemma A.3 in \cite{rogozin2020towards} and is proven in Appendix \ref{app:consensus_iters_strongly_convex}.

	\pd{In the same way we can prove that if the communication network is static, we can establish a sufficiently accurate consensus in the next iteration.
		\begin{lemma}\label{lem:consensus_iters_strongly_convex_fixed}
			Let consensus accuracy be maintained at level $\delta'$, i.e. $\E\norm{\mU^j - \ol\mU^j}^2\le \delta' \text{ for } j = 1, \ldots, k$ and let the communication network be static. Then it is sufficient to make $T_k = T = \sqrt{\chi}\log\frac{D}{\delta'}$ consensus iterations, where $D$ is defined in \eqref{eq:def_sqrt_D}, in order to ensure $\delta'$-accuracy on step $k+1$, i.e. $\E\norm{\mU^{k+1} - \ol\mU^{k+1}}^2\le \delta'$.
		\end{lemma}
	}
	
	\subsubsection{Putting the proof together}
	
	\rev{We derive the expressions for $r$ and $\delta$ to meet the requirement $\E f(\ol x^k) - f(x^*) \le \eps$ according to \eqref{eq:inexact_agd_function_residual}. This is done by combining the results of Lemmas \ref{lem:Ak_properties}, \ref{lem:consensus_iters_strongly_convex}, \ref{lem:consensus_iters_strongly_convex_fixed}.} The details are given in Appendix \ref{app:total_iterations_strongly_convex}.

	%\todo{In the final bound we need to substitute $\sigma$ from \eqref{eq:tilde_g_properties}.}
	
	%\todo{Add accelerated consensus.}
	
	% \section{Numerical experiments}
	
	% We consider the logistic regression problem with L2 regularizer:
	% \begin{align*}
	%     f(x) = \frac{1}{n}\sum\limits_{i=1}^n \log\left(1 + \exp(-b_i\angles{a_i, x})\right) + \frac{\theta}{2}\norm{x}_2^2.
	% \end{align*}
	% Here $a_1,\ldots,a_n\in\R^d$ denote the data points, $b_1,\ldots,b_n\in\{-1, 1\}$ denote class labels and $\theta > 0$ is a penalty coefficient.
	
	% Also we run experiments on a least-squares task:
	% \begin{align*}
	%     f(x) = \frac{1}{2}\norm{\mA x - b}_2^2.
	% \end{align*}
	% The blocks of data matrix $\mA$ and vector $b$ are distributed among the agents in the network.
	
	% The simulations are run on LIBSVM datasets \cite{Chang2011}. We compare the performance of Algorithm \ref{alg:decentralized_agd} (DAccGD in legend of plots) to EXTRA \cite{shi2015extra}, DIGing \cite{Nedic2017achieving}, Mudag \cite{ye2020multi} and APM-C \cite{li2020decentralized}.
	
	% Logistic regression is carried out on a9a data-set, inner iterations are set to $T = 5$ for Mudag and DAccGD. Generation of random geometric graph goes on $20$ nodes.

	\section{Numerical tests}
	
	We run Algorithm \ref{alg:decentralized_agd} on L2-regularized logistic regression problem:
	\begin{align*}
	f(x) = \frac{1}{m}\sum\limits_{i=1}^m \log\left(1 + \exp(-b_i\angles{a_i, x})\right) + \frac{\theta}{2}\norm{x}^2.
	\end{align*}
	Here $a_1,\ldots,a_m\in\R^d$ are entries of the dataset, $b_1,\ldots,b_m\in\{-1, 1\}$ denote class labels and $\theta > 0$ is a penalty coefficient. Data points $(a_i, b_i)$ are distributed among the computational nodes in the network.
	
	We use LIBSVM datasets \cite{Chang2011} to run our experiments. Work of Algorithm \ref{alg:decentralized_agd} is simulated on a9a data-set with different settings for batch-size $r$ and number of consensus iterations $T$. The random geometric graph has $20$ nodes. \rev{We compare the performance of Algorithm 1 with DSGD \cite{fallah2019robust,olshevsky2019asymptotic,olshevsky2019non}}.
	
	\begin{figure}[H]
		\includegraphics[width=0.5\textwidth]{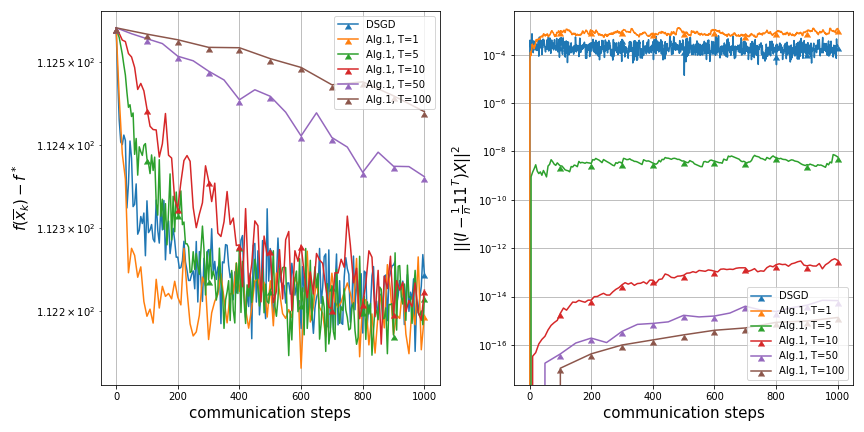}
		\caption{Random geometric graph with 20 nodes; batch size $r = 10$}
	\end{figure}

	% \begin{figure}[H]
	%     \includegraphics[width=0.5\textwidth]{figures/a9a-20bs-5,10,30cons.png}
	%     \caption{Static graph with accelerated consensus; $r = 20$}
	% \end{figure}
	% \begin{figure}[H]
	%     \includegraphics[width=0.5\textwidth]{figures/a9a-50bs-5,10,30cons.png}
	%     \caption{Static graph with accelerated consensus; $r = 50$}
	% \end{figure}
	% \begin{figure}[H]
	%     \includegraphics[width=0.5\textwidth]{figures/a9a-200bs-5,10,30cons.png}
	%     \caption{Static graph with accelerated consensus; $r = 200$}
	% \end{figure}
	
	We observe a tradeoff between consensus accuracy and convergence speed in function value. A large number of consensus steps results in more accurate consensus and slower convergence, and vice versa. This tradeoff is present for different batch sizes.

	\section{Conclusion}
	
	\pd{We propose an accelerated distributed optimization algorithm for stochastic optimization problems in two settings: time-varying graphs and static graphs. For the latter setting we achieve the full acceleration and our method achieves lower bounds both for the communication and oracle per node complexity. }
	
	\pd{
		Our approach is based on accelerated gradient method with stochastic inexact oracle which makes it 
		generic with many possible extensions. In particular, we focus on a specific case of strongly convex smooth functions, but the possible extensions include non-strongly convex and/or non-smooth functions that can be covered by such inexact oracles \cite{devolder2014first,kamzolov2020universal}. Further, we believe that our results can be extended for composite optimization problems, zeroth-order optimization methods \cite{gorbunov2018accelerated,vorontsova2019accelerated,dvurechensky2021accelerated}, and distributed algorithms for saddle-point problems \cite{gasnikov2021accelerated} and variational inequalities.}
	
	% Another interesting application of inexact oracle approach lies in stochastic decentralized algorithms. Consider a class of $L$-smooth $\mu$-strongly convex objectives with gradient noise of each $f_i$ being upper-bounded by $\sigma^2$. For this class of problems, lower complexity bounds write as \cite{arjevani2015communication}
	% \begin{align*}
	%     O(\sqrt{L/\mu\chi}\log(1/\eps)) &\qquad\text{stochastic oracle calls per node} \\
	%     O\cbraces{\max\cbraces{\frac{\sigma^2}{m\mu\eps}, \sqrt{L/\mu}}\log(1/\eps)} &\qquad\text{communication steps}.
	% \end{align*}
	% At the moment, there exist methods which are optimal either in the number of oracle calls or in the number of communication steps. We believe that approach of this paper combined with a specific batch-size choice described in \cite{dvinskikh2020accelerated} allows to develop a decentralized algorithm reaching both optimal complexities.

	%\bibliographystyle{unsrt}
	\bibliographystyle{IEEEtran.bst}
	\bibliography{references,auxiliary}

\begin{thebibliography}{10}
\providecommand{\url}[1]{#1}
\csname url@rmstyle\endcsname
\providecommand{\newblock}{\relax}
\providecommand{\bibinfo}[2]{#2}
\providecommand\BIBentrySTDinterwordspacing{\spaceskip=0pt\relax}
\providecommand\BIBentryALTinterwordstretchfactor{4}
\providecommand\BIBentryALTinterwordspacing{\spaceskip=\fontdimen2\font plus
\BIBentryALTinterwordstretchfactor\fontdimen3\font minus
  \fontdimen4\font\relax}
\providecommand\BIBforeignlanguage[2]{{%
\expandafter\ifx\csname l@#1\endcsname\relax
\typeout{** WARNING: IEEEtran.bst: No hyphenation pattern has been}%
\typeout{** loaded for the language `#1'. Using the pattern for}%
\typeout{** the default language instead.}%
\else
\language=\csname l@#1\endcsname
\fi
#2}}

\bibitem{bor82}
V.~Borkar and P.~P. Varaiya, ``Asymptotic agreement in distributed
  estimation,'' \emph{IEEE Transactions on Automatic Control}, vol.~27, no.~3,
  pp. 650--655, 1982.

\bibitem{tsi84}
J.~N. Tsitsiklis and M.~Athans, ``Convergence and asymptotic agreement in
  distributed decision problems,'' \emph{IEEE Transactions on Automatic
  Control}, vol.~29, no.~1, pp. 42--50, 1984.

\bibitem{deg74}
M.~H. DeGroot, ``Reaching a consensus,'' \emph{Journal of the American
  Statistical Association}, vol.~69, no. 345, pp. 118--121, 1974.

\bibitem{xia06}
L.~Xiao and S.~Boyd, ``{Optimal scaling of a gradient method for distributed
  resource allocation},'' \emph{Journal of Optimization Theory and
  Applications}, vol. 129, no.~3, pp. 469--488, 2006.

\bibitem{rab04}
M.~Rabbat and R.~Nowak, ``Decentralized source localization and tracking
  wireless sensor networks,'' in \emph{Proceedings of the IEEE International
  Conference on Acoustics, Speech, and Signal Processing}, vol.~3, 2004, pp.
  921--924.

\bibitem{ram2009distributed}
S.~S. Ram, V.~V. Veeravalli, and A.~Nedic, ``Distributed non-autonomous power
  control through distributed convex optimization,'' in \emph{IEEE INFOCOM
  2009}.\hskip 1em plus 0.5em minus 0.4em\relax IEEE, 2009, pp. 3001--3005.

\bibitem{kra13}
T.~Kraska, A.~Talwalkar, J.~C. Duchi, R.~Griffith, M.~J. Franklin, and M.~I.
  Jordan, ``Mlbase: A distributed machine-learning system.'' in \emph{CIDR},
  vol.~1, 2013, pp. 2--1.

\bibitem{ned17e}
A.~Nedi{\'c}, A.~Olshevsky, and C.~A. Uribe, ``Distributed learning for
  cooperative inference,'' \emph{arXiv preprint arXiv:1704.02718}, 2017.

\bibitem{nedic2017fast}
------, ``Fast convergence rates for distributed non-bayesian learning,''
  \emph{IEEE Transactions on Automatic Control}, vol.~62, no.~11, pp.
  5538--5553, 2017.

\bibitem{uribe2018distributed}
C.~A. {Uribe}, D.~{Dvinskikh}, P.~{Dvurechensky}, A.~{Gasnikov}, and
  A.~{Nedi\'c}, ``Distributed computation of {W}asserstein barycenters over
  networks,'' in \emph{2018 IEEE Conference on Decision and Control (CDC)},
  2018, pp. 6544--6549.

\bibitem{kroshnin2019complexity}
A.~Kroshnin, N.~Tupitsa, D.~Dvinskikh, P.~Dvurechensky, A.~Gasnikov, and
  C.~Uribe, ``On the complexity of approximating {W}asserstein barycenters,''
  in \emph{Proceedings of the 36th International Conference on Machine
  Learning}.\hskip 1em plus 0.5em minus 0.4em\relax PMLR, 2019, pp. 3530--3540.

\bibitem{ivanova2020composite}
\BIBentryALTinterwordspacing
A.~Ivanova, P.~Dvurechensky, A.~Gasnikov, and D.~Kamzolov, ``Composite
  optimization for the resource allocation problem,'' \emph{Optimization
  Methods and Software}, 2020. [Online]. Available:
  \url{https://doi.org/10.1080/10556788.2020.1712599}
\BIBentrySTDinterwordspacing

\bibitem{bot10}
L.~Bottou, ``Large-scale machine learning with stochastic gradient descent,''
  in \emph{Proceedings of COMPSTAT'2010}.\hskip 1em plus 0.5em minus
  0.4em\relax Springer, 2010, pp. 177--186.

\bibitem{boy11}
S.~Boyd, N.~Parikh, E.~Chu, B.~Peleato, and J.~Eckstein, ``Distributed
  optimization and statistical learning via the alternating direction method of
  multipliers,'' \emph{Foundations and Trends{\textregistered} in Machine
  Learning}, vol.~3, no.~1, pp. 1--122, 2011.

\bibitem{aba16}
M.~Abadi, A.~Agarwal, P.~Barham, E.~Brevdo, Z.~Chen, C.~Citro, G.~S. Corrado,
  A.~Davis, J.~Dean, M.~Devin, \emph{et~al.}, ``Tensorflow: Large-scale machine
  learning on heterogeneous distributed systems.'' in \emph{Conf. on Language
  Resources and Evaluation (LREC’08)}, 2016, pp. 3243--3249.

\bibitem{ned16w}
A.~Nedi{\'c}, A.~Olshevsky, and W.~Shi, ``Achieving geometric convergence for
  distributed optimization over time-varying graphs,'' \emph{SIAM Journal on
  Optimization}, vol.~27, no.~4, pp. 2597--2633, 2017.

\bibitem{ned15}
A.~Nedi{\'c}, A.~Olshevsky, and C.~A. Uribe, ``{Fast convergence rates for
  distributed \mbox{non-Bayesian} learning},'' \emph{IEEE Transactions on
  Automatic Control}, vol.~62, no.~11, pp. 5538--5553, Nov 2017.

\bibitem{ned09}
A.~Nedi\'{c}, A.~Olshevsky, A.~Ozdaglar, and J.~N. Tsitsiklis, ``On distributed
  averaging algorithms and quantization effects,'' \emph{IEEE Transactions on
  Automatic Control}, vol.~54, no.~11, pp. 2506--2517, 2009.

\bibitem{ram10}
S.~S. Ram, A.~Nedi{\'c}, and V.~V. Veeravalli, ``Distributed stochastic
  subgradient projection algorithms for convex optimization,'' \emph{Journal of
  Optimization Theory and Applications}, vol. 147, no.~3, pp. 516--545, 2010.

\bibitem{daneshmand2021newton}
A.~Daneshmand, G.~Scutari, P.~Dvurechensky, and A.~Gasnikov, ``Newton method
  over networks is fast up to the statistical precision,'' in \emph{Proceedings
  of the 38th International Conference on Machine Learning}, vol. 139.\hskip
  1em plus 0.5em minus 0.4em\relax PMLR, 2021, pp. 2398--2409.

\bibitem{scaman2017optimal}
K.~Scaman, F.~Bach, S.~Bubeck, Y.~T. Lee, and L.~Massouli{\'e}, ``Optimal
  algorithms for smooth and strongly convex distributed optimization in
  networks,'' in \emph{International Conference on Machine Learning}, 2017, pp.
  3027--3036.

\bibitem{ghadimi2013optimal}
S.~Ghadimi and G.~Lan, ``Optimal stochastic approximation algorithms for
  strongly convex stochastic composite optimization, ii: Shrinking procedures
  and optimal algorithms,'' \emph{SIAM Journal on Optimization}, vol.~23,
  no.~4, pp. 2061--2089, 2013.

\bibitem{dvurechensky2018parallel}
P.~E. Dvurechensky, A.~V. Gasnikov, and A.~A. Lagunovskaya, ``Parallel
  algorithms and probability of large deviation for stochastic convex
  optimization problems,'' \emph{Numerical Analysis and Applications}, vol.~11,
  no.~1, pp. 33--37, 2018.

\bibitem{Zhang2018}
Y.~Zhang and L.~Xiao, \emph{Communication-Efficient Distributed Optimization of
  Self-concordant Empirical Loss}.\hskip 1em plus 0.5em minus 0.4em\relax Cham:
  Springer International Publishing, 2018, pp. 289--341.

\bibitem{dvurechensky2021hyperfast}
P.~Dvurechensky, D.~Kamzolov, A.~Lukashevich, S.~Lee, E.~Ordentlich, C.~A.
  Uribe, and A.~Gasnikov, ``Hyperfast second-order local solvers for efficient
  statistically preconditioned distributed optimization,''
  \emph{arXiv:2102.08246}, 2021.

\bibitem{agafonov2021accelerated}
A.~Agafonov, P.~Dvurechensky, G.~Scutari, A.~Gasnikov, D.~Kamzolov,
  A.~Lukashevich, and A.~Daneshmand, ``An accelerated second-order method for
  distributed stochastic optimization,'' in \emph{2021 60th IEEE Conference on
  Decision and Control (CDC)}, 2021.

\bibitem{dvinskikh2019decentralized}
D.~Dvinskikh and A.~Gasnikov, ``Decentralized and parallel primal and dual
  accelerated methods for stochastic convex programming problems,''
  \emph{Journal of Inverse and Ill-posed Problems}, vol.~29, no.~3, pp.
  385--405, 2021.

\bibitem{Nemirovskii1983}
A.~Nemirovskii and Yudin, \emph{Problem Complexity and Method Efficiency in
  Optimization}.\hskip 1em plus 0.5em minus 0.4em\relax Wiley, 1983.

\bibitem{nesterov1983method}
Y.~Nesterov, ``A method of solving a convex programming problem with
  convergence rate $o(1/k^2)$,'' \emph{Soviet Mathematics Doklady}, vol.~27,
  no.~2, pp. 372--376, 1983.

\bibitem{nesterov2020primal-dual}
\BIBentryALTinterwordspacing
Y.~Nesterov, A.~Gasnikov, S.~Guminov, and P.~Dvurechensky, ``Primal-dual
  accelerated gradient methods with small-dimensional relaxation oracle,''
  \emph{Optimization Methods and Software}, pp. 1--28, 2020. [Online].
  Available: \url{https://doi.org/10.1080/10556788.2020.1731747}
\BIBentrySTDinterwordspacing

\bibitem{guminov2019accelerated}
S.~V. Guminov, Y.~E. Nesterov, P.~E. Dvurechensky, and A.~V. Gasnikov,
  ``Accelerated primal-dual gradient descent with linesearch for convex,
  nonconvex, and nonsmooth optimization problems,'' \emph{Doklady Mathematics},
  vol.~99, no.~2, pp. 125--128, 2019.

\bibitem{fallah2019robust}
A.~Fallah, M.~Gurbuzbalaban, A.~Ozdaglar, U.~Simsekli, and L.~Zhu, ``Robust
  distributed accelerated stochastic gradient methods for multi-agent
  networks,'' \emph{arXiv preprint arXiv:1910.08701}, 2019.

\bibitem{dvinskikh2019primal}
D.~Dvinskikh, E.~Gorbunov, A.~Gasnikov, P.~Dvurechensky, and C.~A. Uribe, ``On
  primal and dual approaches for distributed stochastic convex optimization
  over networks,'' in \emph{2019 IEEE 58th Conference on Decision and Control
  (CDC)}, 2019, pp. 7435--7440.

\bibitem{olshevsky2019asymptotic}
A.~Olshevsky, I.~C. Paschalidis, and S.~Pu, ``Asymptotic network independence
  in distributed optimization for machine learning,'' \emph{arXiv preprint
  arXiv:1906.12345}, 2019.

\bibitem{olshevsky2019non}
------, ``A non-asymptotic analysis of network independence for distributed
  stochastic gradient descent,'' \emph{arXiv:1906.02702}, 2019.

\bibitem{Nedic2009}
A.~Nedic and A.~Ozdaglar, ``Distributed subgradient methods for multi-agent
  optimization,'' \emph{IEEE Transactions on Automatic Control}, vol.~54,
  no.~1, pp. 48--61, 2009.

\bibitem{shi2015extra}
W.~Shi, Q.~Ling, G.~Wu, and W.~Yin, ``Extra: An exact first-order algorithm for
  decentralized consensus optimization,'' \emph{SIAM Journal on Optimization},
  vol.~25, no.~2, pp. 944--966, 2015.

\bibitem{rogozin2019projected}
A.~Rogozin and A.~Gasnikov, ``Projected gradient method for decentralized
  optimization over time-varying networks,'' \emph{arXiv:1911.08527}, 2019.

\bibitem{Qu2017}
G.~{Qu} and N.~{Li}, ``Accelerated distributed nesterov gradient descent,''
  \emph{2016 54th Annual Allerton Conference on Communication, Control, and
  Computing}, 2016.

\bibitem{ye2020multi}
H.~Ye, L.~Luo, Z.~Zhou, and T.~Zhang, ``Multi-consensus decentralized
  accelerated gradient descent,'' \emph{arXiv preprint arXiv:2005.00797}, 2020.

\bibitem{li2018sharp}
H.~Li, C.~Fang, W.~Yin, and Z.~Lin, ``A sharp convergence rate analysis for
  distributed accelerated gradient methods,'' \emph{arXiv:1810.01053}, 2018.

\bibitem{Jakovetic}
D.~Jakovetic, ``A unification and generalization of exact distributed first
  order methods,'' \emph{IEEE Transactions on Signal and Information Processing
  over Networks}, pp. 31--46, 2019.

\bibitem{dvu2018}
P.~Dvurechenskii, D.~Dvinskikh, A.~Gasnikov, C.~Uribe, and A.~Nedich,
  ``Decentralize and randomize: Faster algorithm for wasserstein barycenters,''
  in \emph{Advances in Neural Information Processing Systems 31}, 2018, pp.
  10\,783--10\,793.

\bibitem{li2020revisiting}
H.~Li and Z.~Lin, ``Revisiting extra for smooth distributed optimization,''
  \emph{arXiv preprint arXiv:2002.10110}, 2020.

\bibitem{hendrikx2020optimal}
H.~Hendrikx, F.~Bach, and L.~Massoulie, ``An optimal algorithm for
  decentralized finite sum optimization,'' \emph{arXiv:2005.10675}, 2020.

\bibitem{li2020optimal}
H.~Li, Z.~Lin, and Y.~Fang, ``Optimal accelerated variance reduced extra and
  diging for strongly convex and smooth decentralized optimization,''
  \emph{arXiv preprint arXiv:2009.04373}, 2020.

\bibitem{tang2019practicality}
J.~Tang, K.~Egiazarian, M.~Golbabaee, and M.~Davies, ``The practicality of
  stochastic optimization in imaging inverse problems,''
  \emph{arXiv:1910.10100}, 2019.

\bibitem{Wu2017}
X.~Wu and J.~Lu, ``Fenchel dual gradient methods for distributed convex
  optimization over time-varying networks,'' in \emph{2017 IEEE 56th Annual
  Conference on Decision and Control (CDC)}, 2017, pp. 2894--2899.

\bibitem{Zhang2017}
G.~Zhang and R.~Heusdens, ``Distributed optimization using the primal-dual
  method of multipliers,'' \emph{IEEE Transactions on Signal and Information
  Processing over Networks}, vol.~4, no.~1, pp. 173--187, 2018.

\bibitem{uribe2020dual}
C.~A. Uribe, S.~Lee, A.~Gasnikov, and A.~Nedi{\'c}, ``A dual approach for
  optimal algorithms in distributed optimization over networks,''
  \emph{Optimization Methods and Software}, pp. 1--40, 2020.

\bibitem{Nedic2017achieving}
A.~Nedić, A.~Olshevsky, and W.~Shi, ``Achieving geometric convergence for
  distributed optimization over time-varying graphs,'' \emph{SIAM Journal on
  Optimization}, vol.~27, no.~4, pp. 2597--2633, 2017.

\bibitem{Pu2018}
S.~Pu, W.~Shi, J.~Xu, and A.~Nedich, ``A push-pull gradient method for
  distributed optimization in networks,'' \emph{2018 IEEE Conference on
  Decision and Control (CDC)}, pp. 3385--3390, 2018.

\bibitem{Maros2018}
M.~Maros and J.~Jald{\'e}n, ``Panda: A dual linearly converging method for
  distributed optimization over time-varying undirected graphs,'' \emph{2018
  IEEE Conference on Decision and Control (CDC)}, pp. 6520--6525, 2018.

\bibitem{koloskova2020unified}
A.~Koloskova, N.~Loizou, S.~Boreiri, M.~Jaggi, and S.~U. Stich, ``A unified
  theory of decentralized sgd with changing topology and local updates,''
  \emph{arXiv preprint arXiv:2003.10422}, 2020.

\bibitem{auzinger2011iterative}
A.~Wien, \emph{Iterative solution of large linear systems}.\hskip 1em plus
  0.5em minus 0.4em\relax Lecture Notes, TU Wien, 2011.

\bibitem{scaman2018optimal}
K.~Scaman, F.~Bach, S.~Bubeck, L.~Massouli{\'e}, and Y.~T. Lee, ``Optimal
  algorithms for non-smooth distributed optimization in networks,'' in
  \emph{Advances in Neural Information Processing Systems}, 2018, pp.
  2740--2749.

\bibitem{stonyakin2020inexact}
\BIBentryALTinterwordspacing
F.~Stonyakin, A.~Tyurin, A.~Gasnikov, P.~Dvurechensky, A.~Agafonov,
  D.~Dvinskikh, M.~Alkousa, D.~Pasechnyuk, S.~Artamonov, and V.~Piskunova,
  ``Inexact model: A framework for optimization and variational inequalities,''
  \emph{Optimization Methods and Software}, 2021. [Online]. Available:
  \url{https://doi.org/10.1080/10556788.2021.1924714}
\BIBentrySTDinterwordspacing

\bibitem{devolder2013first}
O.~Devolder, F.~Glineur, and Y.~Nesterov, ``First-order methods with inexact
  oracle: the strongly convex case,'' \emph{CORE DP 2013/16}, 2013.

\bibitem{dvurechensky2016stochastic}
P.~Dvurechensky and A.~Gasnikov, ``Stochastic intermediate gradient method for
  convex problems with stochastic inexact oracle,'' \emph{Journal of
  Optimization Theory and Applications}, vol. 171, no.~1, pp. 121--145, 2016.

\bibitem{gasnikov2016stochasticInter}
A.~V. Gasnikov and P.~E. Dvurechensky, ``Stochastic intermediate gradient
  method for convex optimization problems,'' \emph{Doklady Mathematics},
  vol.~93, no.~2, pp. 148--151, Mar 2016.

\bibitem{rogozin2020towards}
A.~Rogozin, V.~Lukoshkin, A.~Gasnikov, D.~Kovalev, and E.~Shulgin, ``Towards
  accelerated rates for distributed optimization over time-varying networks,''
  \emph{arXiv preprint arXiv:2009.11069}, 2020.

\bibitem{stonyakin2019gradient}
F.~S. Stonyakin, D.~Dvinskikh, P.~Dvurechensky, A.~Kroshnin, O.~Kuznetsova,
  A.~Agafonov, A.~Gasnikov, A.~Tyurin, C.~A. Uribe, D.~Pasechnyuk, and
  S.~Artamonov, ``Gradient methods for problems with inexact model of the
  objective,'' in \emph{Mathematical Optimization Theory and Operations
  Research}, M.~Khachay, Y.~Kochetov, and P.~Pardalos, Eds.\hskip 1em plus
  0.5em minus 0.4em\relax Cham: Springer International Publishing, 2019, pp.
  97--114.

\bibitem{Chang2011}
C.-C. Chang and C.-J. Lin, ``Libsvm: a library for support vector machines,''
  \emph{ACM transactions on intelligent systems and technology (TIST)}, vol.~2,
  no.~3, p.~27, 2011.

\bibitem{devolder2014first}
O.~Devolder, F.~Glineur, and Y.~Nesterov, ``First-order methods of smooth
  convex optimization with inexact oracle,'' \emph{Mathematical Programming},
  vol. 146, no. 1-2, pp. 37--75, 2014.

\bibitem{kamzolov2020universal}
\BIBentryALTinterwordspacing
D.~Kamzolov, P.~Dvurechensky, and A.~V. Gasnikov, ``Universal intermediate
  gradient method for convex problems with inexact oracle,'' \emph{Optimization
  Methods and Software}, 2020. [Online]. Available:
  \url{https://doi.org/10.1080/10556788.2019.1711079}
\BIBentrySTDinterwordspacing

\bibitem{gorbunov2018accelerated}
E.~Gorbunov, P.~Dvurechensky, and A.~Gasnikov, ``An accelerated method for
  derivative-free smooth stochastic convex optimization,''
  \emph{arXiv:1802.09022}, 2018.

\bibitem{vorontsova2019accelerated}
E.~A. Vorontsova, A.~V. Gasnikov, E.~A. Gorbunov, and P.~E. Dvurechenskii,
  ``Accelerated gradient-free optimization methods with a non-euclidean
  proximal operator,'' \emph{Automation and Remote Control}, vol.~80, no.~8,
  pp. 1487--1501, 2019.

\bibitem{dvurechensky2021accelerated}
P.~Dvurechensky, E.~Gorbunov, and A.~Gasnikov, ``An accelerated directional
  derivative method for smooth stochastic convex optimization,'' \emph{European
  Journal of Operational Research}, vol. 290, no.~2, pp. 601 -- 621, 2021.

\bibitem{gasnikov2021accelerated}
A.~V. Gasnikov, D.~M. Dvinskikh, P.~E. Dvurechensky, D.~I. Kamzolov, V.~V.
  Matyukhin, D.~A. Pasechnyuk, N.~K. Tupitsa, and A.~V. Chernov, ``Accelerated
  meta-algorithm for convex optimization problems,'' \emph{Computational
  Mathematics and Mathematical Physics}, vol.~61, no.~1, pp. 17--28, 2021.

\end{thebibliography}
	
	\newpage
	\onecolumn
	\section*{APPENDIX}
\parindent0pt

\subsection{Proof of Lemma \ref{lem:inexact_oracle}}\label{app:inexact_oracle}

\begin{proof}
	The first statement is proved in Lemma 2.1 of \cite{rogozin2020towards}. For the second statement, we have
	\begin{align*}
	\mathbb{E} \tild{g}_{\delta, L, \mu}(\ol x, \mX) 
	&= \frac{1}{n}\sum_{i=1}^n  \frac{1}{r} \sum_{j=1}^r \mathbb{E} \nabla\bbf_i(x_i,\xi_i^j)
	= \frac{1}{n}\sum_{i=1}^n  \frac{1}{r} \sum_{j=1}^r \nabla f_i(x_i)
	= \frac{1}{n}\sum_{i=1}^n \nabla f_i(x_i) = g_{\delta, L, \mu}(\ol x, \mX).
	\end{align*}
	It remains to show \eqref{eq:tilde_g_properties_2}.
	\begin{align*}
	&\E \norm{\tild{g}_{\delta, L, \mu}(\ol x, \mX) - g_{\delta, L, \mu}(\ol x, \mX)}^2
	\leq \E \norm{ \frac{1}{n}\sum_{i=1}^n  \frac{1}{r} \sum_{j=1}^r \nabla\bbf_i(x_i,\xi_i^j) -  \frac{1}{n}\sum_{i=1}^n \nabla f_i(x_i) }^2 \\
	&\qquad \leq \frac{1}{n^2} \sum_{i=1}^n \E \norm{ \frac{1}{r} \sum_{j=1}^r \nabla\bbf_i(x_i,\xi_i^j) - \nabla f_i(x_i) }^2
	\leq \frac{1}{n^2 r^2} \sum_{i=1}^n \sum_{j=1}^r \E \norm{ \nabla\bbf_i(x_i,\xi_i^j) - \nabla f_i(x_i) }^2 \\
	&\qquad \leq \frac{\sum_{i=1}^n\sigma^2_i}{n^2r}\pd{=
		\frac{\sigma_g^2}{nr}}.
	\end{align*}
	The last inequality follows directly from (\ref{stoch_assumption_on_variance}).
\end{proof}

\subsection{Proof of Theorem \ref{Th:fast_str_conv_adap}}\label{app:fast_str_conv_adap}

For proving the theorem about complexity bounds, we need the following auxiliary Lemma:
\begin{lemma}
\label{lem:opt_subproblem}
Let $\psi$ be convex function. Then for
\begin{equation}
    y = \argmin_{x \in Q} \left( \psi(x) + \beta \norm{x-z}^2 + \gamma \norm{x-u}^2 \right),
\end{equation}
where $\beta > 0$ and $\gamma > 0$, the following is true for any $x \in Q$:
\begin{equation}
    \psi(x) + \beta \norm{x-z}^2 + \gamma \norm{x-u}^2 \geq \psi(y) + \beta \norm{y-z}^2 + \gamma \norm{y-u}^2 + (\beta + \gamma) \norm{x-y}^2.
\end{equation}
\end{lemma}
\begin{proof}
    As $y$ is minimum, the subgradient of function at point $y$ includes $0$:
    \begin{equation*}
        \exists g: g + \beta \nabla_x \norm{x-z}^2 \vert_{x=y} + \gamma \nabla_x \norm{x-u}^2 \vert_{x=y} = 0.
    \end{equation*}
    It holds
    \begin{align}
        \label{eq:opt_subproblem}
        \psi(x) - \psi(y) &\geq \angles{g, x-y} = \angles{\beta \nabla_x \norm{x-z}^2 \vert_{x=y} + \gamma \nabla_x \norm{x-u}^2 \vert_{x=y}, y - x} \\
        &= \angles{2 \beta (y - z) + 2 \gamma (y - u), y - x}
    \end{align}
    and we get that
    \begin{align*}
        2 \angles{y-z,y-x} &= \norm{y}^2 - \norm{z}^2 - 2 \angles{z, y-z}+\norm{x}^2 - \norm{y}^2 - 2 \angles{y, x-y} - \norm{x}^2 + \norm{z}^2 + 2 \angles{z,x-z} \\
        &= \norm{y-z}^2 + \norm{x-y}^2 - \norm{x-z}^2.
    \end{align*}
    After similar manipulations with the $2 \angles{y-u,y-x}$ term and replacing the right part in \eqref{eq:opt_subproblem}, the lemma statement is obtained.
\end{proof}

Now we pass to the proof of Theorem \ref{Th:fast_str_conv_adap} itself.

\begin{proof}
    %The proof is mostly based on proof of Theorem 3.4 from \cite{stonyakin2020inexact}. 
    We begin from the right inequality from \eqref{eq:inexact_oracle_def_devolder}:
    $$f(y) - \cbraces{f_{\delta_{k+1},L,\mu}(x) + \angles{g_{\delta_{k+1},L,\mu}(x), y - x}} \le \frac{L}{2} \norm{y - x}^2 + \delta_{k+1}.$$
    It can be rewritten as
     \begin{align*}
         f(x^{k+1}) &- \cbraces{f_{\delta_{k+1},L,\mu}(y^{k+1}) + \angles{\tild{g}^{\ r_{k+1}}_{\delta_{k+1},L,\mu}(y^{k+1}), x^{k+1} - y^{k+1}}} \\ &\leq \angles{g_{\delta_{k+1},L,\mu}(y^{k+1}) - \tild{g}^{\ r_{k+1}}_{\delta_{k+1},L,\mu}(y^{k+1}), x^{k+1} - y^{k+1}}  + \frac{L}{2} \norm{y^{k+1} - x^{k+1}}^2 + \delta_{k+1}.
     \end{align*}
    The first term in the right hand side can be estimated using Young inequality:
    \begin{align*}
        \angles{g_{\delta_{k+1},L,\mu}(y^{k+1}) - \tild{g}^{\ r_{k+1}}_{\delta_{k+1},L,\mu}(y^{k+1}), x^{k+1} - y^{k+1}} \leq \frac{L}{2} \norm{x^{k+1} - y^{k+1}}^2 + \frac{1}{2L} \norm{g_{\delta_{k+1},L,\mu}(y^{k+1}) - \tild{g}^{\ r_{k+1}}_{\delta_{k+1},L,\mu}(y^{k+1})}^2
    \end{align*}
    The combination of the last two expressions yields
    \begin{align*}
         f(x^{k+1}) \leq f_{\delta_{k+1},L,\mu}(y^{k+1}) + \angles{\tild{g}^{\ r_{k+1}}_{\delta_{k+1},L,\mu}(y^{k+1}), x^{k+1} - y^{k+1}} + L \norm{y^{k+1} - x^{k+1}}^2 + \lambda + \delta_{k+1},
    \end{align*}
    where 
    $$
    \lambda := \frac{1}{2L}\norm{g_{\delta_{k+1},L,\mu}(y^{k+1}) - \tild{g}^{\ r_{k+1}}_{\delta_{k+1},L,\mu}(y^{k+1})}^2
    $$ 
    for brevity. We substitute $x^{k+1}, y^{k+1}$ into several terms by their definitions:
    \begin{align*}
         &f(x^{k+1}) \leq f_{\delta_{k+1},L,\mu}(y^{k+1}) + \angles{\tild{g}^{\ r_{k+1}}_{\delta_{k+1},L,\mu}(y^{k+1}), \frac{\alpha^{k+1} u^{k+1} + A^k x^k}{A^{k+1}} - y^{k+1}} \\ &+ L \norm{\frac{\alpha^{k+1}u^k + A^k x^k}{A^{k+1}} - \frac{\alpha^{k+1} u^{k+1} + A^k x^k}{A^{k+1}}}^2 + \lambda + \delta_{k+1}.
    \end{align*}
    As $A^{k+1} = A^k + \alpha^{k+1}$ by definition and as dot product is convex, we get the following:
    \begin{align*}
        f(x^{k+1}) &\leq \frac{A^k}{A^{k+1}} \left( f_{\delta_{k+1}, L, \mu} (y^{k+1}) + \angles{\tild{g}^{\ r_{k+1}}_{\delta_{k+1}, L, \mu} (y^{k+1}), x^k - y^{k+1} }\right) \\  
        &\quad+ \frac{\alpha^{k+1}}{A^{k+1}} \left( f_{\delta_{k+1}, L, \mu} (y^{k+1}) + \angles{\tild{g}^{\ r_{k+1}}_{\delta_{k+1}, L, \mu} (y^{k+1}), u^{k+1} - y^{k+1}} \right) +  \frac{L (\alpha^{k+1})^2}{(A^{k+1})^2} \norm{u^k - u^{k+1}}^2 + \lambda + \delta_{k+1}.
    \end{align*}
    By definition of $\alpha^{k+1}$ we have:
    \begin{align*}
        f(x^{k+1}) &\leq \frac{A^k}{A^{k+1}} \left( f_{\delta_{k+1}, L, \mu} (y^{k+1}) + \angles{\tild{g}^{\ r_{k+1}}_{\delta_{k+1}, L, \mu} (y^{k+1}), x^k - y^{k+1} }\right) \\  
        &\quad + \frac{\alpha^{k+1}}{A^{k+1}} \left( f_{\delta_{k+1}, L, \mu} (y^{k+1}) + \angles{\tild{g}^{\ r_{k+1}}_{\delta_{k+1}, L, \mu} (y^{k+1}), u^{k+1} - y^{k+1}} \right) +  \frac{1 + A^k \mu}{A^{k+1}} \norm{u^k - u^{k+1}}^2 + \lambda + \delta_{k+1}.
    \end{align*}
    We rewrite that as:
    \begin{align*}
        f(x^{k+1}) &\leq \frac{A^k}{A^{k+1}} \left( f_{\delta_{k+1}, L, \mu} (y^{k+1}) + \angles{g_{\delta_{k+1}, L, \mu} (y^{k+1}), x^k - y^{k+1}} + \angles{\tild{g}^{\ r_{k+1}}_{\delta_{k+1}, L, \mu}(y^{k+1}) - g_{\delta_{k+1}, L, \mu} (y^{k+1}), x^k - y^{k+1}} \right) \\  
        &\quad+ \frac{\alpha^{k+1}}{A^{k+1}} \left( f_{\delta_{k+1}, L, \mu} (y^{k+1}) + \angles{\tild{g}^{\ r_{k+1}}_{\delta_{k+1}, L, \mu} (y^{k+1}), u^{k+1} - y^{k+1}} \right) +  \frac{1 + A^k \mu}{A^{k+1}} \norm{u^k - u^{k+1}}^2 + \lambda + \delta_{k+1}.
    \end{align*}
    Using left part of (\ref{eq:inexact_oracle_def_devolder}), we get:
    \begin{align}
    \label{eq:convergency_proof_intermediate}
        f(x^{k+1}) &\leq \frac{A^k}{A^{k+1}} \left( f
        (x^k) + \angles{\tild{g}^{\ r_{k+1}}_{\delta_{k+1}, L, \mu}(y^{k+1}) - g_{\delta_{k+1}, L, \mu} (y^{k+1}), x^k - y^{k+1}} \right) \\  &+ \frac{\alpha^{k+1}}{A^{k+1}} \left( f_{\delta_{k+1}, L, \mu} (y^{k+1}) + \angles{\tild{g}^{\ r_{k+1}}_{\delta_{k+1}, L, \mu} (y^{k+1}), u^{k+1} - y^{k+1}} \right) +  \frac{1 + A^k \mu}{A^{k+1}} \norm{u^k - u^{k+1}}^2 + \lambda + \delta_{k+1}.
    \end{align}
    From lemma \ref{lem:opt_subproblem} for optimization problem at step 5 in Algorithm \ref{alg:FastAlg2_strong} we have:
    \begin{align*}
        \alpha^{k+1} &\angles{\tild{g}^{\ r_{k+1}}_{\delta_{k+1}, L, \mu}  (y^{k+1}) ,u^{k+1}-y^{k+1}} + (1 + A^k \mu) \norm{u^{k+1} - u^k}^2 + \alpha^{k+1} \mu \norm{u^{k+1} - y^{k+1}}^2 \\
        &+ (1 + A^k \mu + \alpha^{k+1} \mu ) \norm{u^{k+1} - x}^2  \leq \alpha^{k+1} \angles{\tild{g}^{\ r_{k+1}}_{\delta_{k+1}, L, \mu}  (y^{k+1}) ,x-y^{k+1}} \\ &+ (1 + A^k \mu) \norm{x - u^k}^2 +  \alpha^{k+1} \mu \norm{x - y^{k+1}}^2.
    \end{align*}
    As squared norm is always non-negative, we obtain
    \begin{align}
    \label{eq:convergency_proof_lemma}
        \alpha^{k+1} &\angles{\tild{g}^{\ r_{k+1}}_{\delta_{k+1}, L, \mu}  (y^{k+1}) ,u^{k+1}-y^{k+1}} + (1 + A^k \mu) \norm{u^{k+1} - u^k}^2 
        \leq -(1 + A^k \mu + \alpha^{k+1} \mu ) \norm{u^{k+1} - x}^2 \\  &+ \alpha^{k+1} \angles{\tild{g}^{\ r_{k+1}}_{\delta_{k+1}, L, \mu}  (y^{k+1}) ,x-y^{k+1}} + (1 + A^k \mu) \norm{x - u^k}^2 + \alpha^{k+1} \mu \norm{x - y^{k+1}}^2
    \end{align}
    Combining inequalities \eqref{eq:convergency_proof_intermediate} and \eqref{eq:convergency_proof_lemma}, we get:
    \begin{align*}
        A^{k+1} f(x^{k+1}) \leq~&A^k \left( f (x^k) + \angles{\tild{g}^{\ r_{k+1}}_{\delta_{k+1}, L, \mu}(y^{k+1}) - g_{\delta_{k+1}, L, \mu} (y^{k+1}), x^k - y^{k+1}} \right) \\
        &+ \alpha^{k+1} \left( f_{\delta_{k+1}, L, \mu} (y^{k+1}) + \angles{\tild{g}^{\ r_{k+1}}_{\delta_{k+1}, L, \mu}  (y^{k+1}) ,x-y^{k+1}} + \mu \norm{x - y^{k+1}}^2 \right) \\ 
        &+ (1 + A^k \mu) \norm{x - u^k}^2 - (1 + A^k \mu + \alpha^{k+1} \mu ) \norm{u^{k+1} - x}^2 + A^{k+1} \lambda + \delta_{k+1} A^{k+1} \\
        =~&A^k \left( f (x^k) + \angles{\tild{g}^{\ r_{k+1}}_{\delta_{k+1}, L, \mu}(y^{k+1}) - g_{\delta_{k+1}, L, \mu} (y^{k+1}), x^k - y^{k+1}} \right) \\
        &+ \alpha^{k+1} \left( f_{\delta_{k+1}, L, \mu} (y^{k+1}) + \angles{{g}_{\delta_{k+1}, L, \mu}  (y^{k+1}) ,x-y^{k+1}} \right) \\ 
        &+ \alpha^{k+1} \left( \angles{\tild{g}^{\ r_{k+1}}_{\delta_{k+1}, L, \mu}  (y^{k+1}) - {g}_{\delta_{k+1}, L, \mu}  (y^{k+1}) ,x-y^{k+1}} + \mu \norm{x - y^{k+1}}^2 \right) \\ &+ (1 + A^k \mu) \norm{x - u^k}^2 - (1 + A^k \mu + \alpha^{k+1} \mu ) \norm{u^{k+1} - x}^2 + A^{k+1} \lambda + \delta_{k+1} A^{k+1}.
    \end{align*}
    Using the left part of \eqref{eq:inexact_oracle_def_devolder} again results in
    \begin{align*}
        A^{k+1} f(x^{k+1}) \leq &{A^k} \left( f (x^k) + \angles{\tild{g}^{\ r_{k+1}}_{\delta_{k+1}, L, \mu}(y^{k+1}) - g_{\delta_{k+1}, L, \mu} (y^{k+1}), x^k - y^{k+1}} \right) \\
        &+ \alpha^{k+1} \left( f (x)+ \angles{\tild{g}^{\ r_{k+1}}_{\delta_{k+1}, L, \mu}  (y^{k+1}) - {g}_{\delta_{k+1}, L, \mu}  (y^{k+1}) ,x-y^{k+1}} + \right) \\ &+ (1 + A^k \mu) \norm{x - u^k}^2 - (1 + A^k \mu + \alpha^{k+1} \mu) \norm{u^{k+1} - x}^2 + A^{k+1} \lambda + \delta_{k+1} A^{k+1}.
    \end{align*}
    We can take expectation and we may see that angles terms go zero as $\E \tild{g}^{\ r}_{\delta, L, \mu}(x) = g_{\delta, L, \mu}(x)$ for any $r$, and $\lambda \leq \frac{\sigma^2}{2Lr_{k+1}}$:
    \begin{align*}
        A^{k+1} \E f(x^{k+1}) &- {A^k}  f (x^k) \leq \alpha^{k+1} f (x)  \\ &+ (1 + A^k \mu) \norm{x - u^k}^2 - (1 + A^{k+1} \mu ) \E \norm{u^{k+1} - x}^2 + \frac{\sigma^2 A^{k+1}}{2Lr_{k+1}} + \delta_{k+1} A^{k+1}.
    \end{align*}
    Now we should pay attention to the fact that the expectation is \textit{conditional} because we consider $x^k$ and other $k$-th variables known before the iteration:
    \begin{align*}
        A^{k+1} \E \left[ f(x^{k+1}) | x^k, \ldots, x^1 \right] &- {A^k}  f (x^k) \leq \alpha^{k+1} f (x)  + (1 + A^k \mu) \norm{x - u^k}^2 \\ &- (1 + A^{k+1} \mu ) \E \left[ \norm{u^{k+1} - x}^2 | x^k, \ldots, x^1 \right] + \frac{\sigma^2 A^{k+1}}{2Lr_{k+1}} + \E[\delta_{k+1} | x^k, \ldots, x^1] A^{k+1}.
    \end{align*}
    If we take $x = x^*$, write these inequalities for all $k$ from $0$ to $N-1$ and sum up all of them, we will get the following:
    \begin{align*}
    \sum_{i=1}^N A^i &\mathbb{E} \left[ f(x^i) | x^i, \ldots, x^1 \right] \leq \sum_{i=0}^{N-1} A^i f(x^i) + \sum_{i=1}^N \alpha_i f(x^*) + \\
    &+ \sum_{i=0}^{N-1} (1 + A^i \mu) \norm{u^i - x^*}^2 - \sum_{i=1}^N (1 + A^i \mu) \mathbb{E} \left[ \norm{u^i - x^*}^2 | x^i, \ldots, x^1 \right] +  \sum_{i=1}^N A_i\left( \frac{\sigma^2}{2Lr_{i}} + \E[\delta_i | x^i, \ldots, x^1] \right).
    \end{align*}
    Next, we use the law of total expectation $N$ times and get rid of conditional expectations, and after that get rid of similar terms:
    \begin{gather}
    \label{eq:agd_bound}
    \mathbb{E} A^N f(x^N) \leq A^N f(x^*) + \norm{u^0 - x^*}^2 - (1 + A^N \mu) \mathbb{E} \norm{u^N - x^*}^2 + 
    \sum_{i=1}^N A_i \left( \frac{\sigma^2}{2Lr_{i}} + \E\delta_i \right)
    \end{gather}
    Here we also recall that $A_0 = \alpha_0 = 0, \sum_{i=1}^N \alpha_i = A_N$.
    \par Finally, we get
    $$ \mathbb{E} f(x^N) - f(x^*) \leq \frac{1}{A^N} \left( \norm{u^0 - x^*}^2 +  \sum_{i=1}^N A^i \left( \frac{\sigma^2}{2Lr_{i}} + \E\delta_i \right) \right).
    $$
    The second inequality is obtained from \eqref{eq:agd_bound} and the fact that $f(x) \geq f(x^*)$. 
\end{proof}

\subsection{Proof of Lemma \ref{lem:Ak_properties}}\label{app:Ak_properties}

\begin{proof}
In view of definition of sequence $\alpha^{k+1}$, we have:
\begin{align*}
A^N &\leq A^N(1 + \mu A^{N-1}) = L(A^N - A^{N-1})^2\\
&\leq L(\sqrt{A^N} - \sqrt{A^{N-1}})^2(\sqrt{A^N} + \sqrt{A^{N-1}})^2 \leq 4L^{N} A^N (\sqrt{A^N} - \sqrt{A^{N-1}})^2.
\end{align*}
For the case when $\mu > 0$ we obtain:
\begin{align*}
\mu A^{N-1} A^{N} \leq A^N(1 + \mu A^{N-1}) \leq 4L A^N (\sqrt{A^N} - \sqrt{A^{N-1}})^2.
\end{align*}
From the fact that $A^1 = 1 / L$ and the last inequality we can show that
\begin{align*}
\sqrt{A^N} \geq \left(1 + \frac{1}{2}\sqrt{\frac{\mu}{L}}\right)\sqrt{A^{N-1}} \geq \frac{1}{\sqrt{L}}\left(1 + \frac{1}{2}\sqrt{\frac{\mu}{L}}\right)^{(N-1)}.
\end{align*}
For the second statement, we recall the proof of Lemma A.1 in \cite{rogozin2020towards}.
Update rule for $A^k$ writes as
	\begin{align}\label{eq:coef_sequence}
	1 + \mu A^k = \frac{L(\alpha^{k+1})^2}{A^{k+1}},~ A^{k} = \sum_{i=0}^k \alpha^{i},~ \alpha^0 = 0.
	\end{align}
	A sequence $\braces{B^k}_{k=0}^\infty$ with a similar update rule is studied in \cite{devolder2013first}.
	\begin{align}\label{eq:nesterov_coef_sequence}
	L + \mu B^k = \frac{L (\beta^{k+1})^2}{B^{k+1}},~ B^k = \sum_{i=0}^k \beta^i,~ \beta^0 = 1,
	\end{align}
	and for sequence $\braces{B^k}_{k=0}^\infty$ it is shown $\frac{\sum_{i=0}^k B^i}{B^k}\le 1 + \sqrt{L / \mu}$.
	Dividing \eqref{eq:nesterov_coef_sequence} by L yields
	\begin{align*}
	1 + \mu(B^k / L) = \frac{L (\beta^{k+1} / L)^2}{(B^{k+1} / L)},
	\end{align*}
	which means that update rule for $B^k / L$ is equivalent to \eqref{eq:coef_sequence}. Since $A^1 = 1 / L = B^0 / L$, it holds $A^{k+1} = B^k / L,~ k\ge 0$ and
	\begin{align*}
	\frac{\sum_{i=1}^k A^i}{A^k} = \frac{\sum_{i=0}^{k-1} B^i / L}{B^{k-1} / L} \le 1 + \sqrt\frac{L}{\mu}.
	\end{align*}
\end{proof}

\subsection{Proof of Lemma \ref{lem:consensus_iters_strongly_convex}}\label{app:consensus_iters_strongly_convex}

\begin{proof}
	The proof follows by revisiting proof of Lemma A.3 in \cite{rogozin2020towards} in stochastic setting. First, note that multiplication by a mixing matrix does not change the average of a vector, i.e. $\frac{1}{n}\onevec\onevec^\top\mX = \frac{1}{n}\onevec\onevec^\top \mW^k \mX~ \forall k\ge 0$. This means $\ol\mU^{k+1} = \ol\mV^{k+1}$.

    Second, let us use the contraction property of mixing matrix sequence $\{\mW^k\}_{k=0}^\infty$. We have
	\begin{align*}
	\E\norm{\mU^{k+1} - \aU^{k+1}}^2 
	&\le (1 - \lambda)^{2(T/\tau)} \E\norm{\mV^{k+1} - \aU^{k+1}}^2 
	\le e^{-2(T/\tau)\lambda} \E\norm{\mV^{k+1} - \aU^{k+1}}^2. \\
	\end{align*}
	Assuming that $\E\norm{\mV^{k+1} - \aU^{k+1}}^2\le D$, we only need $T = \frac{\tau}{2\lambda}\log\frac{D}{\delta'}$ iterations to ensure $\E\norm{\mU^{k+1} - \aU^{k+1}}^2\le\delta'$. In the rest of the proof, we show that $\E\norm{\mV^{k+1} - \aU^{k+1}} = \E\norm{\mV^{k+1} - \aV^{k+1}} \le \sqrt D$.
	
    According to update rule of Algorithm \ref{alg:decentralized_agd}, it holds
	\begin{align*}
	\E\norm{\mV^{k+1} - \ol\mV^{k+1}} 
	&\le \frac{\alpha^{k+1}\mu \E\norm{\mY^{k+1} - \ol\mY^{k+1}}}{1 + A^{k+1}\mu} + \frac{(1 + A^k\mu)\E\norm{\mU^k - \ol\mU^k}}{1 + A^{k+1}\mu} + \frac{\alpha^{k+1}}{1 + A^{k+1}\mu} \E\norm{\nabla^r F(\mY^{k+1})} \\
	&\le \sqrt{\delta'} + \frac{\alpha^{k+1}}{1 + A^{k+1}\mu} \E\norm{\nabla^r F(\mY^{k+1})}.
	\end{align*}
	We estimate $\norm{\nabla^r F(\mY^{k+1})}$ using $\Lworst$-smoothness of $\nabla F$: 
	\begin{align*}
	\norm{\nabla^r F(\mY^{k+1})} 
	&\le \norm{\nabla^r F(\mY^{k+1}) - \nabla^r F(\mX^*)} + \norm{\nabla^r F(\mX^*)} \\
	&\le \Lworst\underbrace{\norm{\mY^{k+1} - \ol\mY^{k+1}}}_{\le\sqrt{\delta'}} + \Lworst\underbrace{\norm{\ol\mY^{k+1} - \mX^*}}_{= \sqrt n \norm{\ol y^{k+1} - x^*}} + \norm{\nabla^r F(\mX^*)} \numberthis\label{eq:gradeint_y_upper_bound}
	\end{align*}
	where $x^* = \argmin_{x\in\R^d} f(x),~ \mX^* = \onevec (x^*)^\top$. It remains to estimate $\norm{\ol y^{k+1} - x^*}$.
	\begin{align*}
	\norm{\ol y^{k+1} - x^*}
	&\le \frac{\alpha^{k+1}}{A^{k+1}} \norm{\ol x^{k+1} - x^*} + \frac{A^{k}}{A^{k+1}}\norm{\ol u^{k+1} - x^*}
	\le \max\braces{\norm{\ol x^{k+1} - x^*}, \norm{\ol u^{k+1} - x^*}}
	\end{align*}
	
	By Lemma \ref{lem:inexact_agd_convergence} and strong convexity of $f$:
	\begin{align*}
	\E\norm{\ol x^{k+1} - x^*}^2 
	&\le \frac{2}{\mu}\cbraces{\E f(\ol x^{k+1}) - f(x^*)}
	\le \frac{\norm{\ol u^0 - x^*}^2}{A^{k+1}\mu} + \frac{\sum_{i=1}^{k+1} A^i}{A^{k+1}\mu} \cbraces{\frac{\sigma^2}{2Lr} + \delta}
	\end{align*}
	and therefore
	\begin{align*}
	\E\norm{\ol y^{k+1} - x^*}^2 
	&\le \max\braces{\frac{\norm{\ol u^0 - x^*}^2}{A^{k+1}\mu} + \frac{\sum_{i=1}^{k+1} A^i}{A^{k+1}\mu}\cbraces{\frac{\sigma^2}{2Lr} + \delta},~ \frac{\norm{\ol u^0 - x^*}^2}{1 + A^{k+1}\mu} + \frac{\sum_{i=1}^{k+1} A^i}{1 + A^{k+1}\mu}\cbraces{\frac{\sigma^2}{2Lr} + \delta}} \\
	&\le \frac{\norm{\ol u^0 - x^*}^2}{A^{k+1}\mu} + \frac{1}{\mu} \cbraces{1 + \sqrt\frac{L}{\mu}} \cbraces{\frac{\sigma^2}{2Lr} + \delta},
	\end{align*}
	where the last inequality holds by Lemma \ref{lem:Ak_properties}.
	
	Returning to \eqref{eq:gradeint_y_upper_bound}, we get
	\begin{align*}
	&\norm{\nabla^r F(\mY^{k+1})} \\
	&\qquad\le \Lworst\sqrt{\delta'} + \Lworst\sqrt{n} \cbraces{\frac{\norm{\ol u^0 - x^*}^2}{A^{k+1}\mu} + \frac{1}{\mu} \cbraces{1 + \sqrt\frac{L}{\mu}} \cbraces{\frac{\sigma^2}{2Lr} + \delta}}^{1/2} + \norm{\nabla^r F(\mX^*)} \\
	&\qquad\le \Lworst\sqrt{\delta'} + \Lworst\sqrt{n} \cbraces{\frac{L}{\mu} \norm{\ol u^0 - x^*}^2 \cbraces{1 + \sqrt\frac{\mu}{2L}}^{-2k} + \frac{2L^{1/2}}{\mu^{3/2}} \cbraces{\frac{\sigma^2}{2Lr} + \delta}}^{1/2} + \norm{\nabla^r F(\mX^*)} \\
	&\qquad\le \Lworst\sqrt{\delta'} + \Lworst\sqrt{n} \cbraces{\frac{L}{\mu} \norm{\ol u^0 - x^*}^2 + \frac{2L^{1/2}}{\mu^{3/2}} \cbraces{\frac{\sigma^2}{2Lr} + \delta}}^{1/2} + \norm{\nabla^r F(\mX^*)}.
	\end{align*}
	For distance to consensus of $\mV^{k+1}$, it holds
	\begin{align*}
	\E\norm{\mV^{k+1} - \aV^{k+1}} \le \sqrt{\delta'} + \frac{\alpha^{k+1}}{1 + A^k\mu + \mu} \E\norm{\nabla F\cbraces{\mY^{k+1}}}
	\end{align*}
	We estimate coefficient by $\E\norm{\nabla F\cbraces{\mY^{k+1}}}$ using the definition of $\alpha^{k+1}$.
	\begin{align*}
	&1 + A^k\mu = \frac{L(\alpha^{k+1})^2}{A^k + \alpha^{k+1}} \\
	&L(\alpha^{k+1})^2 - (1 + A^k\mu)\alpha^{k+1} - (1 + A^k\mu)A^k = 0 \\
	&\alpha^{k+1} = \frac{1+ A^k\mu + \sqrt{(1 + A^k\mu)^2 + 4LA^k(1 + A^k\mu)}}{2L} \\
	&\frac{\alpha^{k+1}}{1 + A^{k+1}\mu} \le \frac{\alpha^{k+1}}{1 + A^k\mu} = \frac{1}{2L} \cbraces{1 + \sqrt{1 + 4\frac{LA^k}{1 + A^k\mu}}} \\
	&\qquad\le \frac{1}{2L}\cbraces{\sqrt{\frac{L}{\mu}} + \sqrt{\frac{L}{\mu} + 4\frac{L}{\mu}}} \le \frac{2}{\sqrt{L\mu}}
	\end{align*}
	Returning to $\mV^{k+1}$, we get
	\begin{align*}
	&\E\norm{\mV^{k+1} - \aV^{k+1}} \\
	&\qquad\le \cbraces{\frac{2\Lmax}{\sqrt{L\mu}} + 1}\sqrt{\delta'} + \Lmax\sqrt{\frac{n}{L\mu}} \cbraces{\frac{L}{\mu} \norm{\ol u^0 - x^*}^2 + \frac{2L^{1/2}}{\mu^{3/2}} \cbraces{\frac{\sigma^2}{2Lr} + \delta}}^{1/2} + \frac{2\E\norm{\nabla^r F(\mX^*)}}{\sqrt{L\mu}} \\
	&\qquad\leq \cbraces{\frac{2\Lmax}{\sqrt{\Lav\muav}} + 1}\sqrt{\delta'} + \frac{2\Lmax}{\muav} \sqrt{n} \cbraces{\norm{\ol u^0 - x^*}^2 + \frac{2}{\sqrt{\Lav\muav}}\cbraces{\frac{\sigmaav^2}{4n\Lav r^2} + \delta}}^{1/2} + \frac{2nM_\xi}{\sqrt{\Lav\muav}} = \sqrt D,
	\end{align*}
	where in the last inequality we used $\norm{\nabla^r F(\mX^*)} \leq n M_\xi$.
\end{proof}

\subsection{Putting the proof of Theorem \ref{th:total_iterations_strongly_convex} together}\label{app:total_iterations_strongly_convex}

Let us show that choice of number of subroutine iterations $T_k = T$ yields
\begin{align*}
\E f(\ol x^k) - f(x^*) \le \frac{1}{A^k}\cbraces{\norm{\ol u^0 - x^*}^2 + \cbraces{\frac{\sigma^2}{2Lr} + \delta} \sum_{i=1}^k A^i}
\end{align*}
by induction. At $k=0$, we have $\norm{\mU^0 - \aU^0} = 0$ and by Lemma \ref{lem:inexact_agd_convergence} it holds
\begin{align*}
\E f(\ol x^1) - f(x^*) &\le \frac{1}{A^1}\cbraces{\norm{\ol u^0 - x^*}^2 + \cbraces{\frac{\sigma^2}{2Lr} + \delta} A^1}.
\end{align*}
For induction pass, assume that $\E\norm{\mU^j - \aU^j}^2\le \delta'$ for $j = 0,\ldots, k$. By Lemma \ref{lem:consensus_iters_strongly_convex}, if we set $T_k = T$, then $\E\norm{\mU^{k+1} - \aU^{k+1}}^2\le \delta'$. Applying Lemma \ref{lem:inexact_agd_convergence} again, we get
\begin{align*}
\E f(\ol x^k) - f(x^*) \le \frac{1}{A^k}\cbraces{\norm{\ol u^0 - x^*}^2 + \cbraces{\frac{\sigma^2}{2Lr} + \delta} \sum_{i=1}^k A^i}.
\end{align*}	

Next, we substitute a bound on $A^k$ from Lemma \ref{lem:Ak_properties} and get
	\begin{align*}
	\E f(\ol x^N) &- f(x^*) \\
	&~\le LR^2 \cbraces{1 + \frac{1}{2}\sqrt\frac{\mu}{L}}^{-2(N-1)} + \cbraces{1 + \sqrt\frac{L}{\mu}} \cbraces{\frac{\sigma^2}{2Lr} + \delta} \\
	&~ = 2\Lav R^2\cbraces{1 + \frac{1}{4}\sqrt\frac{\muav}{\Lav}}^{-2(N-1)} \hspace{-0.3cm} + \cbraces{1 + 2\sqrt\frac{\Lav}{\muav}} \cbraces{\frac{\sigma^2}{4\Lav r} + \delta}.
	\end{align*}
%	\begin{align*}
%	\boxed{
%		q = \Lav,~ p = \frac{\muav}{2}\longrightarrow~ L = 2\Lav,~ \mu = \frac{\muav}{2}.
%	}
%	\end{align*}
It remains to estimate the number of iterations required for $\eps$-accuracy. In order to satisfy
\begin{align*}
2\Lav\norm{\ol u^0 - x^*}^2\cbraces{1 + \frac{1}{4}\sqrt\frac{\muav}{\Lav}}^{-2(N-1)} &\le \frac{\eps}{2}, \\
\cbraces{1 + 2\sqrt\frac{\Lav}{\muav}} \cbraces{\frac{\sigma^2}{4\Lav r} + \delta} &\le \frac{\eps}{2},
\end{align*}
it is sufficient to choose $\delta' = \dfrac{n\eps}{32} \dfrac{\muav^{3/2}}{\Lav^{1/2} \Lmax^2},~ r = \dfrac{2\sigma^2}{\eps\sqrt{\Lav\muav}}$ and
\begin{align*}
N = 3\sqrt{\frac{\Lav}{\muav}} \log\cbraces{\frac{4\Lav\norm{\ol u^0 - x^*}^2}{\eps}}.
\end{align*}
\pd{Finally, the total number of stochastic oracle calls per node equals
	\begin{align*}
	N_{orcl} = N\cdot r = \dfrac{6\sigmaav^2}{n\muav\eps} \log\cbraces{\frac{4\Lav\norm{\ol u^0 - x^*}^2}{\eps}}.
	\end{align*}
	Further, the total number of communications is
	\begin{align*}
	N_{\text{comm}} 
	&= N\cdot T = 3\sqrt\frac{\Lav}{\muav} \log\cbraces{\frac{4\Lav\norm{\ol u^0 - x^*}^2}{\eps}}\cdot \kappa \cdot \log\frac{D}{\delta'} \\
	&= O\left(\sqrt{\frac{\Lav}{\muav}} \kappa \cdot \log\cbraces{\frac{4\Lav\norm{\ol u^0 - x^*}^2}{\eps}} \log\frac{D}{\delta'}\right),
	\end{align*}
	where $\kappa = \frac{\tau}{2\lambda}$ if the communication network is time-varying and $\kappa = \sqrt{\chi}$ if the communication network is fixed.}

%Finally, the total number of communications is
%\begin{align*}
%N_{\text{comm}} 
%&= N\cdot T = 3\sqrt\frac{\Lav}{\muav} \log\cbraces{\frac{4\Lav\norm{\ol u^0 - x^*}^2}{\eps}}\cdot \frac{\tau}{2\lambda}\cdot \log\frac{D}{\delta'} \\
%&= \frac{3}{2}\sqrt{\frac{\Lav}{\muav}} \frac{\tau}{\lambda} \cdot \log\cbraces{\frac{4\Lav\norm{\ol u^0 - x^*}^2}{\eps}} \log\frac{D}{\delta'},
%\end{align*}
%and the number of stochastic oracle calls per node equals
%\begin{align*}
%N_{comp} = N\cdot r = \dfrac{6\sigmaav^2}{n\muav\eps} \log\cbraces{\frac{4\Lav\norm{\ol u^0 - x^*}^2}{\eps}}.
%\end{align*}

\end{document}